\documentclass[10pt,a4paper]{article}
\usepackage[utf8]{inputenc}
\usepackage{amsmath}
\usepackage{amsthm}
\usepackage{amsfonts}
\usepackage{amssymb, amsbsy}
\usepackage{color}
\usepackage{hyperref}
\usepackage{enumitem}
\usepackage{comment}

\usepackage[a4paper, textwidth=13.6cm]{geometry}

\newcommand{\ueps}{u_{\epsilon}}
\newcommand{\uepsin}{u_{\epsilon}^{\rm{in}}}
\newcommand{\tuepsin}{\tilde{u}_{\epsilon}^{\rm{in}}}
\newcommand{\uin}{u^{\rm{in}}}
\newcommand{\repsin}{R_{\epsilon}^{\rm{in}}}
\newcommand{\trepsin}{\tilde{R}_{\epsilon}^{\rm{in}}}
\newcommand{\rin}{R^{\rm{in}}}
\newcommand{\seps}{S_{\epsilon}}
\newcommand{\ujeps}{u^j_{\epsilon}}
\newcommand{\rjeps}{R^j_{\epsilon}}

\newcommand{\ubteps}{\overline{U}_{\dt}^{\epsilon}}
\newcommand{\uhteps}{\widehat{U}_{\dt}^{\epsilon}}
\newcommand{\rbteps}{\overline{R}_{\dt}^{\epsilon}}
\newcommand{\rhteps}{\widehat{R}_{\dt}^{\epsilon}}
\newcommand{\jbteps}{\overline{J}_{\dt}^{\epsilon}}
\newcommand{\jhteps}{\widehat{J}_{\dt}^{\epsilon}}
\newcommand{\wbteps}{\overline{W}_{\dt}^{\epsilon}}
\newcommand{\whteps}{\widehat{W}_{\dt}^{\epsilon}}

\newcommand{\gbteps}{G_{\dt}^{\epsilon}}
\newcommand{\sbteps}{S_{\dt}^{\epsilon}}
\newcommand{\fbteps}{F_{\dt}^{\epsilon}}
\newcommand{\dbteps}{D_{\dt}^{\epsilon}}
\newcommand{\vbteps}{V_{\dt}^{\epsilon}}

\newcommand{\ujoeps}{u^{j-1}_{\epsilon}}
\newcommand{\rjoeps}{R^{j-1}_{\epsilon}}
\newcommand{\jjeps}{J^j_{\epsilon}}
\newcommand{\jjoeps}{J^{j-1}_{\epsilon}}

\newcommand{\feps}{F_{\epsilon}}
\newcommand{\fieps}{F_{\epsilon}^{-1}}
\newcommand{\fiteps}{F_{\epsilon}^{-T}}
\renewcommand{\oe}{\Omega_{\epsilon}}
\newcommand{\tueps}{\tilde{u}_{\epsilon}}
\newcommand{\treps}{\tilde{R}_{\epsilon}}
\newcommand{\jeps}{J_{\epsilon}}
\newcommand{\veps}{V_{\epsilon}}
\newcommand{\deps}{D_{\epsilon}}
\newcommand{\oet}{\Omega_{\epsilon}(t)}
\newcommand{\get}{\Gamma_{\epsilon}(t)}
\newcommand{\geps}{\Gamma_{\epsilon}}
\newcommand{\reps}{R_{\epsilon}}
\newcommand{\m}{\mathfrak{m}}

\newcommand{\peps}{\phi_{\epsilon}}
\newcommand{\R}{\mathbb{R}}
\newcommand{\N}{\mathbb{N}}
\newcommand{\Z}{\mathbb{Z}}
\newcommand{\oR}{\overline{R}}
\newcommand{\uR}{\underline{R}}

\newcommand{\ie}{i.\,e.,\,}
\newcommand{\fxe}{\frac{x}{\epsilon}}
\newcommand{\dt}{\Delta t}

\newcommand{\kl}{<}
\newcommand{\gr}{>}
\newcommand{\weps}{w_{\epsilon}}
\newcommand{\dth}{\partial_t^h}
\newcommand{\J}{\bar{J}}
\newcommand{\per}{\mathrm{per}}

\newcommand{\heps}{h_{\epsilon}}
\newcommand{\tweps}{\widetilde{w}_{\epsilon}}
\newcommand{\feg}{\frac{1}{\vert \epsilon \Gamma \vert}}
\newcommand{\bfxe}{\left[\fxe\right]}
\newcommand{\teps}{\mathcal{T}_{\epsilon}}

\newcommand*\Bell{\ensuremath{\boldsymbol\ell}}
\newcommand{\Roe}{\Omega_{\epsilon}}

\newtheorem{definition}{Definition}
\newtheorem{remark}{Remark}
\newtheorem{theorem}{Theorem}
\newtheorem{proposition}{Proposition}
\newtheorem{lemma}{Lemma}
\newtheorem{corollary}{Corollary}

\title{Homogenization of a mineral dissolution and precipitation model involving free boundaries at the micro scale}
\author{Markus Gahn and Iuliu Sorin Pop}

\begin{document}

\maketitle

\begin{abstract}
In this work we present the homogenization of a reaction-diffusion model that includes an evolving microstructure. Such type of problems model, for example, mineral dissolution and precipitation in a porous medium. Hence, we are dealing with a multi-scale problem with free boundaries on the pore scale. In the initial state the microscopic geometry is given by a periodically perforated domain, including spherical solid grains. The radius of each grain is of order $\epsilon$ and  depends on the unknown (the solute concentration) at its surface. Therefore the radii of the grains change in time, leading to a nonlinear, free boundary problem. In a first step, we transform the evolving micro domain to a fixed, periodically domain. Using the Rothe-method we prove the existence of a weak solution and obtain \textit{a priori} estimates that are uniform with respect to $\epsilon$. Finally, letting $\epsilon \to 0$, we derive a macroscopic model, the solution of which approximates the micro-scale solution. For this, we use the method of two-scale convergence, and obtain strong compactness results enabling to pass to the limit in the nonlinear terms. 

\end{abstract}

\section{Introduction}

In this paper, we consider a reaction-diffusion model defined in a perforated domain having a heterogeneous microstructure. Due to the reactions taking place at the boundaries of the perforations, these perforations may evolve. This evolution is not known \textit{a priori}, but depends on the solution of the problem, and therefore the model involves free boundaries at the micro scale. The initial state of the domain is isomorphic to a periodically perforated domain, where the periodicity is of order $\epsilon$, with $\epsilon$ being a small, positive scale separation parameter. The aim of this work is to establish existence of a weak solution with $\epsilon$-uniform \textit{a priori} estimates, and  to derive the corresponding macroscopic model. This is obtained by homogenization methods, after passing to the limit $\epsilon \to 0$. The macroscopic model is derived rigorously, its solution being an approximation of the solution to the microscopic model.

Reactive transport in evolving porous media occur in a variety of real-life applications. We mention here mineral precipitation and dissolution \cite{Bringedal2014, Noorden2009, Schulzetal2016}, biofilm growth \cite{SchulzKnabner}, colloid deposition \cite{eden2021multiscale}, or water diffusion into absorbent particles \cite{FasanoMikelic, SWEIJEN2017407}. A typical example is the precipitation and dissolution taking place in a porous medium, consisting of alternating solid grains and voids (the pore space). We assume that the void space is completely filled by a fluid, say, water, which is stationary. Soluble species can diffuse inside the fluid, and can precipitate at the fluid-solid interface to form a solid layer (e.g. salt). The reverse process of dissolution is also possible, by which solute species is released back to the fluid. Assuming that, in general, the precipitate layer has a thickness that is comparable to the typical pore size, when defining the micro-domain one has to exclude the grains and (if present) the precipitate layer. In other words, the boundary of the micro-domain wherein the problem is defined has two parts: the outer part of the boundary, where the entire medium is embedded into, and the inner boundary, defined as the solid-fluid interface. 
One cannot assume that the micro-domain remains fixed in time, as the thickness of the precipitate layer depends on the species concentration, which, itself, is an unknown in the model. Hence, the precipitation and dissolution processes lead to a variable pore space, and the fluid-solid interface is a free boundary. 

In a simplifying setting, we assume here that the porous medium includes spherical solid grains, having a radius of order $\epsilon$, and being periodically-distributed. Furthermore, we assume that the evolution of this grain is radially symmetric, therefore its evolution is well described by the radius of the resulting solid.  In this way, the free boundary is reduced to a one-dimensional equation for the radius. Nevertheless, this still leads to a strongly nonlinear problem, defined in time-dependent domains with \textit{a priori} unknown evolution.

In what follows let $T >0$ stand for a maximal time and consider a heterogeneous medium $\Omega \in \R^d$ ($d \in \{2, 3\}$), consisting of regions occupied by a stationary fluid (the pore space) and of small, solid regions of spherical shape (the grains with the attached mineral layers). These centres of the solid regions are distributed periodically in a $d$-dimensional (hyper)cube structure. The distance between two successive centres is of $\epsilon$ order (the micro-scale length). In this way, the pore space depends on both $\epsilon$ and on the time $t$, and for given a $t \in [0, T]$ we denote the pore space by $\oet$. Within $\oet$, we consider a reaction-diffusion equation, with the solute concentration $\ueps$ as unknown quantity. For the ease of presentation, and since the aim here is to provide mathematically rigorous derivation of the macroscopic model, only one solute species is considered. In the case of two species, the model can be reduced to the situation here by considering a decoupled component of the model, the total electric charge (see, e.g. \cite{KnabnerDuijn,DuijnPop}). 

Recalling the simplifying assumptions made above, 
the solid phase consists of spherical regions, including the original solid part of the medium, and the radially symmetric mineral layer. These solids are characterised by a radius $\reps$, which is grain-dependent and also changes in time. More precisely, it is obtained as the solution of an ordinary differential equation, depending in a nonlinear way on the (averaged) solute concentration at the surface of the grain and the radius itself. In this way, the spatial variable $x$ enters in the equation as a parameter, and varies for every microscopic cell. Such a structure can also be found in the models discussed in \cite{eden2021multiscale, FasanoMikelic, Schulzetal2016, SWEIJEN2017407}, but for different types of applications and not in the context of homogenization. 

As already mentioned, the goal of this contribution is to provide a mathematically rigorous derivation of the macroscopic model approximating the microscopic precipitation-dissolution model defined in evolving microscopic geometries. In doing so, two aspects are essential. First, we transform the problem defined in the evolving micro-domain $\oet$ into one problem defined in a fixed, periodically perforated domain $\oe$. In doing so, we employ the Hanzawa transformation \cite{Hanzawa81} (we also refer to \cite{PruessSimonettBook16} for an overview of this topic). This leads to a change in the coefficients of the equations, which now depend on the radius of the grains and therefore on the unknown concentration. We prove the existence of a solution pair ($\ueps, \reps$) by using the Rothe-method, and derive \textit{a priori} estimates that are uniform with respect to the parameter $\epsilon$. 

Secondly, for the derivation of the macroscopic model, which is obtained in the limit along a sequence $\epsilon \to 0$, we use the method of two-scale convergence. To pass to the limit in the nonlinear terms, we need strong two-scale convergence results for the concentration $\ueps$, as well as for the coefficients depending on the radius $\reps$  and its time derivative $\partial_t \reps$. For the latter, this is obtained as a consequence of the convergence of the radius function $\reps$, for which a Kolmogorov-type compactness argument is used. As will follow from below, due to the nonlinear character of the problem and because of the coefficients for the time-derivative appearing in the transformed model, we cannot obtain $\epsilon$-uniform \textit{a priori} estimates for $\partial_t \ueps$, what would guarantee the strong two-scale convergence of $\ueps$, see \cite{Gahn, MeirmanovZimin}. We therefore apply an alternative strategy, namely to control the product $\jeps \ueps$, where $\jeps $ denotes the Jacobi determinant of the Hanzawa transformation, and to solve an auxiliary approximation problem to establish the strong convergence of $\ueps$.  Finally, for the strong convergence of $\partial_t \reps $ and to identify the limit equation for the radius, we prove a two-scale compactness result for averaged functions on the oscillating surface of the micro cells. With these convergence results, we are able to pass to the limit in the microscopic model. 

The outcome is a macroscopic model consisting of a reaction-diffusion equation for the macroscopic concentration, defined in the entire $\Omega$, coupled with an ordinary differential equation for the macroscopic radius depending on the macroscopic concentration. The effective parameters like the diffusion or the porosity are obtained by solving cell problems formulated in evolving reference cells, and accounting for the microscopic evolution of the porous medium. The resulting is therefore a strongly coupled, two-scale mathematical model, defined in the entire domain $\Omega$ that does not depend on time. 

The analysis and homogenization of reaction-diffusion models defined in evolving microscopic geometries have been addressed in several publications. Various strategies have been adopted, depending also on the particular geometry considered there. The simplest situation appears in one spatial dimension, or if the pore space of a porous medium is a long but thin strip or tube. Then the evolution of the solid-fluid interface can be described by a free boundary function. In this sense we refer to \cite{vNoordenPop} for the one-dimensional case, where the existence and uniqueness of a solution for a precipitation-dissolution model involving multi-valued dissolution rates is proved, using the fixed-domain transformation proposed in  \cite{Fliert}
. In the multi-dimensional case, we mention \cite{vNoorden, KumarNoordenPop}, where effective models are derived formally by transversal averaging. 

For more general situations, as considered here, level set methods have been employed to describe the evolving microscopic geometry. We refer to \cite{Noorden2009, Bringedal2014, BringedalKumar, Garttner, Schulzetal2016}, where dissolution and precipitation processes are modelled, and to \cite{SchulzKnabner} for a similar approach in modelling biofilm growth in a porous medium. However, the rigorous homogenization is missing. Without entering into details, we mention that one can consider as an alternative to the free boundaries the phase-field approach. Then, the fluid-solid interface is approximated by a narrow diffuse-interface layer. 

Whereas a rich literature exists on the rigorous homogenization of reactive transport processes with nonlinear reaction terms in the bulk domain and the microscopic surface for a complex, but fixed microstructure, see for example \cite{DonatoNguyen_HomogenizationOfDiffusion, Gahn, GahnNeussRadu2017SubstrateChanneling,  Gahn2020b,MielkeReicheltThomas2014}, the results for an evolving microstructure are scarce. Moreover, in most of the papers dealing with such cases, the microstructural evolution  is assumed known \textit{a priori}. In this sense, we mention \cite{Peter07, Peter09}, where homogenized models are derived rigorously for  linear reaction-diffusion-advection problems (also see \cite{Fotouhi}). 
Close to the present work is \cite{GahnNRP}, where the reaction and adsorption/desorption terms are nonlinear, without requiring a radially symmetric microscopic evolution, and the diffusion is low, leading to a different scaling for the gradient of the concentration; however, the evolution is assumed known.
Some general two-scale results for  transformations of locally periodic domains are treated in \cite{Wiedemann}.

In all these works, the models are defined in perforated domains, which resembles well the structure of a porous medium with variable microstructure, but the microscopic evolution is known \textit{a priori}. For the analysis of free boundary models with radially symmetric evolution of the interface we refer to \cite{FasanoMikelic}. There, the existence of a solution is proved for a system involving a nonlinear parabolic problem with radially symmetric perforations, in which the rate of change of the radius of these perforations depends linearly on the saturation. Also, the analysis does not include any homogenization results. Also, the two-scale system considered in \cite{eden2021multiscale} is the homogenized counterpart of a microscopic problem that is similar to the one considered here, but for linear adsorption/desorption rates. The existence of a solution is obtained for sufficiently small times. We also mention \cite{FriedmanHu}, where rigorous upscaling results are proved for the Laplace and the heat equations posed in a domain with a rough/rapidly oscillating boundary. This rough boundary is also a free boundary, as its normal velocity depends (linearly) on the solution, with a time-dependent rate.

Our paper is organized as follows: In Section \ref{SectionTheMicroscopicModel} we formulate the microscopic model on the evolving domain $\oet$ and transform it to the fixed domain $\oe$. In Section \ref{SectionExistence} we prove existence for the microscopic model and establish \textit{a priori} estimates depending explicitly on $\epsilon$. In Section \ref{SectionDerivationMacroModel} we show the two-scale convergence results for the microscopic solution and derive the macroscopic model. A conclusion is given in Section \ref{SectionConclusion}. Some elemental calculations concerning the Hanzawa transformation can be found in die Appendix \ref{SectionAppendixHanzawaReferenceElement}.

\section{The microscopic model}
\label{SectionTheMicroscopicModel}

In what follows $T > 0$ is a finite time, $\epsilon>0$ denotes a small scale separation parameter s.t. $\epsilon^{-1} \in \N$ ($n \in \N_0$ being the spatial dimension), and ${ \Bell} = (\ell_1, \dots, \ell_n) \in \N^n$ is an $n$-tuple of strictly positive natural numbers. Let $\Omega = ({\bf 0},{\Bell}) = (0, \ell_1) \times \dots \times (0, \ell_n)$ and $Y:=(0,1)^n$. The center of $Y$ is denoted by $\m:= \frac12 (1,\ldots,1)\in \R^n$. For the set of $n$-tuples $K_{\epsilon}:= \{ {\bf k}\in \Z^n \, : \, \epsilon (Y+{\bf k}) \subset \Omega\}$ one has 
\begin{equation}\label{eq:K}
\overline{\Omega} = \bigcup_{{\bf k} \in K_{\epsilon}} \epsilon \big(\overline{Y} + {\bf k} \big).
\end{equation}

In other words, we decompose $\overline{\Omega}$ into a set of microscopic ($\epsilon$-sized) hypercubes, two different ones sharing at most one side. Each $x \in \overline{\Omega}$ is contained in a microscopic hypercube having the center 
\begin{equation}\label{eq:m}
	\m_{\epsilon,x} := \epsilon \left(\left[\fxe\right] + \m \right).
\end{equation}

To introduce the evolving geometry we first define $\uR, \oR$ representing the minimal, respective maximal admissible radius of an $n$-dimensional sphere contained in a microscopic hypercube, and satisfying 
\begin{equation} \label{eq:Ru}
	0< \uR < \oR < \frac12.
\end{equation}
Then, for any $t \in [0, T]$ and $x \in \overline{\Omega}$ we associate a radius 
\begin{equation}\label{eq:r}
	r_{\epsilon,t,x} := \reps\left(t,\m_{\epsilon,x}\right),
\end{equation}
where the function $\reps$ will be a solution component of the problem stated below. 

Further, for any $x \in \overline{\Omega}$ a unique ${\bf k} \in K_{\epsilon}$ exists s.t. $x \in \epsilon(Y + {\bf k})$, namely ${\bf k} = \bfxe$. Then, for any $x \in \overline{\Omega}$, with ${\bf k}$ as before  and $t \in [0, T]$, we consider the $n$-dimensional ball, respectively sphere 
\begin{align*}
	B_{\epsilon,{\bf k}}(t):= B_{\reps(t,x)}(\epsilon(\m + {\bf k})), \qquad \geps^{\bf k}(t) :=  \geps\left(t,\epsilon \bfxe\right):= \partial B_{\epsilon, {\bf k}}(t) ,
\end{align*}
both having the radius $r_{\epsilon,t,x} \in (\epsilon \uR, \epsilon \oR) \subset \left(0,\dfrac{\epsilon}{2}\right)$. With this, at any time $t$, the evolving microscopic domain and its inner part of the (freely-moving) boundary are defined as 
\begin{align}
	\label{MicroscopicModel_EvolutionDomain}
	\oe(t):= \Omega \setminus \bigcup_{{\bf k} \in K_{\epsilon}} \overline{B_{\epsilon,{\bf k}}(t)}, \text{ respectively } \get:= \bigcup_{{\bf k} \in K_{\epsilon}} \geps^{\bf k}(t).
\end{align}
Observe that the outer boundary $\partial \Omega$ is fixed. Finally, we define the sets 
\begin{equation}\label{eq:QH}
Q_{\epsilon}^T := \bigcup_{t\in (0,T)} \{t\} \times \oet, \quad 
\text{ and } \quad 
H_{\epsilon}^T := \bigcup_{t \in (0,T)} \{t\} \times \get.
\end{equation}
 We consider the following microscopic problem that is motivated by microscopic models for  precipitation-dissolution processes in a porous medium (see e.g. \cite{vNoordenPop, vNoorden,Noorden2009}).    \\[0.5em]
 {\bf Problem P$_\epsilon$}. Find $\ueps, \reps: \bar{Q}_{\epsilon}^T \rightarrow \R$ solving 
\begin{subequations}\label{MicroscopicModel}
\begin{align}
\partial_t \ueps - \nabla \cdot \big(D\nabla \ueps \big) &= f &\mbox{ in }& Q_{\epsilon}^T,
\\
-D\nabla \ueps \cdot \nu_{\epsilon} &= \partial_t \reps (\ueps - \rho) &\mbox{ on }& H_{\epsilon}^T,
\\
-D\nabla \ueps \cdot \nu_{\epsilon} &= 0 &\mbox{ on }& (0,T)\times \partial \Omega,
\\
\label{MicroscopicModel_ODE}
\partial_t \reps &= G_{\epsilon}(t,x,\ueps , \reps) &\mbox{ in }& Q_{\epsilon}^T,
\\
\label{MicroscopicModel_ICu}
\ueps(0) &= \uepsin &\mbox{ in }& \Omega_{\epsilon}(0),
\\
\label{MicroscopicModel_ICR}
\reps(0) &= \repsin &\mbox{ in }& \Omega_{\epsilon}(0).
\end{align}
Here, $D$ is a diffusion tensor (its properties being mentioned in Assumption \ref{AssumptionDiffusion} below) and $\rho > 0$ is a constant representing the molar density of the precipitate. The unknowns $\ueps$ and $\reps$ represent the concentration of the solute species, respectively the thickness of the precipitate layer around a grain centered at $\epsilon \bfxe$. The vector $\nu_{\epsilon}$ is the unit normal to $\get$ outwards $\oet$, while $\uepsin$ and $\repsin$ are given initial conditions (see Assumptions \ref{AssumptionInitialValueUeps} and \ref{AssumptionInitialValueReps}). The function $G_{\epsilon}$ appearing in \eqref{MicroscopicModel_ODE} is defined as 
\begin{equation}\label{eq:G}
	G_{\epsilon}(t,x,\ueps,\reps) = \frac{\epsilon}{\left|\geps\left(t,\epsilon \bfxe\right)\right|} \int_{\geps\left(t,\epsilon \bfxe\right)} g\left(\ueps(t,z),\frac{\reps(t,z)}{\epsilon}\right) dz .
\end{equation}
\end{subequations}
Observe that the domain $\oe(t)$ depends on the unknown function $\reps$, and therefore this problem involves a (microscopic) free boundary. 
Further, for any $t \in (0, T)$, $G_{\epsilon}(t, \cdot ,\ueps,\reps)$ is constant inside microscopic cells, i.e. inside any $\epsilon (Y  + {\bf k})$ with ${\bf k} \in K_{\epsilon}$. The perforation of the microscopic cell evolves in a radially symmetric way with radius $\reps$. From the application point of view, this is a simplified setup, in the sense that we assume that the precipitate layer around a grain has a uniform,  but unknown thickness.  This is guaranteed by \eqref{MicroscopicModel_ODE}.  
Based on Assumption \ref{AssumptionRbounded}, we will prove that $\reps$ remains between two values $\epsilon \uR$ and $\epsilon \oR$, so  $\geps\left(t,\epsilon \bfxe\right)$  can never touch the boundary of the cell, or  degenerate into a point. 

\begin{remark}
	\label{RemarkRepsStepFunction}
	Observe that \eqref{MicroscopicModel_ODE} is stated for all $t \in (0, T)$ and $x \in \oe(t)$, which implies that $\reps $ is defined in the entire $\oe(t)$. In fact, $\reps$ is constant inside each microscopic cell $Y_{\epsilon,\bf k}(t):=\oe(t) \cap \epsilon(Y  + {\bf k})$, with ${\bf k} \in K_{\epsilon}$. This follows from \eqref{MicroscopicModel_ODE} and the definition of $G_{\epsilon}$ in \eqref{eq:G}. Hence, in the following we extend the function $\reps$ constantly from the perforated micro cell $Y_{\epsilon,\bf k}(t)$ to the whole micro cell $\epsilon(Y + k)$.  Further, as we will see later, the micro cells $Y_{\epsilon,\bf k}(t)$ are included in $\Omega$. Especially, we have $\partial \oe(t) \setminus \geps(t)  = \partial \Omega$, i.e., the outer boundary is fixed and $\reps$ is defined on the whole fixed domain $\Omega$.
	%
	%
\end{remark}

For the functional spaces used below the notations are standard. We only mention that $L^\infty(\Omega,C^{0,1}\left([0,T],\big[\epsilon \uR,\epsilon \oR\big]\right))$ is the set of functions that are $C^{0,1}$ w.r.t. time f.a.e. $x \in \Omega$, s.t. $\reps$ is   bounded by $\uR$ and $\oR$.
\begin{definition}
\label{DefSolutionMicProblem}
The pair of functions $(\ueps,\reps)$ with $\ueps  \in L^2((0,T),H^1(\oe(t))) $  and $\reps \in L^\infty(\Omega,C^{0,1}\left([0,T],\big[\epsilon \uR,\epsilon \oR\big]\right))$ is a weak solution of Problem P$_\epsilon$ if, for all $\peps \in C^1\big(\overline{Q_{\epsilon}^T}\big)$ with $ \peps(T,\cdot) = 0$, it holds that
\begin{align}
\begin{aligned}
\label{FreeBVP_MicroscopicModel_VariationalEquation}
-\int_0^T \int_{\oet} & \ueps \partial_t\peps dx dt + \int_0^T \int_{\oet} D\nabla \ueps \cdot \nabla \peps dx dt 
\\
&= \int_0^T \int_{\oet} f \peps dx dt - \int_0^T \int_{\get} G_{\epsilon}(\ueps,\reps)  (\ueps - \rho) \peps d\sigma dt
\\
&+ \int_{\oe(0)}\uepsin\peps(0) dx + \int_0^T \int_{\get} \partial_t \seps(t,\seps^{-1}) \cdot \nu \ueps \peps d\sigma dt
, 
\end{aligned}
\end{align}
where $\seps $ is defined in $\eqref{DefinitionSeps}$. Furthermore, the function $\reps$ fulfills $\eqref{MicroscopicModel_ODE}$ and the initial condition  $\eqref{MicroscopicModel_ICR}$ almost everywhere.  
\end{definition}
Observe that, in the definition above, the function $\seps$ describes the evolution of the moving surface $\get$. The evolution of the evolving domain $\oet $ is given by $\eqref{MicroscopicModel_EvolutionDomain}$.

\subsection{Transformation to the fixed domain}

As mentioned in the introduction, the present analysis relies on the idea of transforming Problem P$_\epsilon$ (or its weak form stated in Definition \ref{DefSolutionMicProblem}, defined on the evolving domain $\oet$, to a problem defined in a fixed, reference domain $\oe$.  To this aim, we employ the Hanzawa transform, \cite{Hanzawa81} (see also \cite{PruessSimonettBook16} for an overview of this topic). With $\oR$ introduced above, see \eqref{eq:Ru}, the domain $\oe$ is defined by
\begin{align*}
\oe := \Omega \setminus \bigcup_{{\bf k} \in K_{\epsilon}} \overline{B_{\epsilon \oR} \left(\epsilon ({\bf k} + \m)\right)}. 
\end{align*}
Further, we define $\geps := \partial \oe \setminus \partial \Omega$. In other words, the microscopic domain is obtained as the union of scaled and shifted reference elements
\begin{align*}
Y^{\ast} := Y \setminus \overline{B_{\oR}(\m)} = (0,1)^n \setminus \overline{B_{\oR}(\m)} .
\end{align*}
The boundary of $B_{\oR}(\m)$ is denoted by $\Gamma := \partial B_{\oR}(\m)$.

With this, for any $t \in [0,T]$ we consider the Hanzawa transformation $S_{\epsilon}(t, \cdot): \oe \rightarrow \oet$, defined as 
\begin{align}\label{DefinitionSeps}
\seps(t,x):= x + \left(\reps(t,x) - \epsilon \oR \right) \chi_0 \left(\fxe\right) \nu_0\left(\fxe\right),
\end{align}
 where the functions $\chi_0$ and $\nu_0$ are $Y$-periodic extensions to the entire $\R^n$ of the functions defined in Appendix  \ref{SectionAppendixHanzawaReferenceElement}. 
 Since $\reps $ is constant for every cell $\epsilon (Y + {\bf k})$ with ${\bf k} \in K_{\epsilon}$, and $\chi_0 \in C^{\infty}_{\per}(Y)$, the function $\seps$ is (strongly) differentiable with respect to $x$. According to the calculations in Appendix \ref{SectionAppendixHanzawaReferenceElement}, \begin{align*}
\nabla \seps (t,x) = I_n + \left(\frac{\reps - \epsilon \oR }{\epsilon}\right) \nabla_y \left( \chi_0 \nu_0\right)\left(\fxe\right),
\end{align*}
with $I_n \in \R^{n \times n}$ being the identity matrix, and
\begin{align} \label{DefinitionTilde}
0 \kl \left(\frac{\uR}{\oR}\right)^{n-1} \le \det \left(\nabla \seps (t,x)\right) \le 1 + 2\left(\oR - \uR\right).
\end{align}
These estimates are crucial for the homogenization, as the functional determinant of the transformation $\seps$ is bounded away from 0, uniformly with respect to $\epsilon$.  We emphasize that $\seps$ is defined and smooth on the whole domain $\Omega$, but the Jacobi determinant is not positive anymore.

Now, define the functions 
\begin{equation}
\begin{array}{ll}
\tueps: (0,T)\times \oe \rightarrow \R,  &\; \tueps(t,x):=\ueps (t,\seps(t,x)),
\\[0.5em]
\treps: (0,T) \times \Roe \rightarrow [\epsilon \uR , \epsilon \oR], &\; \treps(t,x) := \reps(t,\seps(t,x)).
\end{array}
\end{equation}
From Remark \ref{RemarkRepsStepFunction} we immediately obtain that, in fact, $\treps$ and $\reps$ coincide. However, to better distinguish the equation on the fixed and the evolving domain, $\treps$ will appear in the solution pair  $(\tueps,\treps)$ for the problem formulated in the fixed domain. 

To simplify the writing, the functions defined below will be used frequently in what follows, 
\begin{equation}\label{DefFepsJeps}
\feps(t,x) := \nabla \seps(t,x), \; 
\jeps(t,x) := \det (\feps(t,x)).
\end{equation}
The inverse and the transpose of the inverse of $\feps$ are  denoted by 
\begin{equation}\label{DefFepsInv}
    \fieps(t,x):= \left(\feps(t,x)\right)^{-1}, \quad \fiteps(t,x):= \left(\fieps(t,x)\right)^T.
\end{equation}
After employing the Hanzawa transformation \eqref{DefinitionSeps}, 
\eqref{FreeBVP_MicroscopicModel_VariationalEquation} can be rewritten in terms of the variables $\tueps$ and $\treps$. 
More precisely, for all $\phi \in H^1(\oe)$ and almost every $t \in (0,T)$ it holds that  
\begin{equation}\label{eq:tildeu}
\begin{array}{l}
\int_{\oe} \partial_t \big(\jeps \tueps\big) \phi dx - \int_{\oe} \tueps \partial_t \jeps \phi dx - \int_{\oe} V_{\epsilon} \cdot \nabla \tueps \phi dx \\[0.5em]
\qquad + \int_{\oe} D_{\epsilon}\nabla \tueps \cdot \nabla \phi dx 
 = \int_{\oe} \jeps f_{\epsilon} \phi dx  - \int_{\geps} \partial_t \treps (\tueps - \rho) \jeps \phi d\sigma ,
\end{array}
\end{equation}
 with $f_{\epsilon}, V_{\epsilon}$ and $D_{\epsilon}$ defined as 
\begin{equation}\label{DefVepsDeps}
\begin{array}{rcl}
f_{\epsilon}(t,x) &:=& f(t,\seps(t,x)),
\\[0.5em]
V_{\epsilon}(t,x) &:=& \jeps(t,x) \fiteps(t,x) \partial_t \seps(t,x),
\\[0.5em]
D_{\epsilon}(t,x) &:=& \jeps(t,x) \fieps D(\seps(t,x)) \fiteps(t,x).
\end{array}
\end{equation}
The evolution of the radius $\treps$ is given by
\begin{align*}
\partial_t \treps(t,x) &= G_{\epsilon}(t,x,\tueps , \treps)= \frac{\epsilon}{|\epsilon \Gamma|} \int_{\geps\left(\epsilon \bfxe\right)} g \left( \tueps (t,z),\frac{\treps(t,z)}{\epsilon}\right) d\sigma,
\end{align*}
where $\geps: \R^n \rightarrow \R^n$ is defined as 
\begin{align}
\label{DefinitionGepsx}
\geps(z):= \epsilon \Gamma + z, \text{ for any } z \in \R^n.
\end{align}
The initial conditions read
\begin{align*}
\tueps(0) &= \uepsin(\seps(0,\cdot)) =: \tuepsin \quad \mbox{ in } \oe,
\\
\treps(0) &= \repsin(\seps(0,\cdot)) =: \trepsin \quad \mbox{ in } \oe.
\end{align*}
As before, one has $\trepsin = \repsin$ since $\repsin$ is constant on every cell $\epsilon (Y + {\bf k})$, ${\bf k} \in K_{\epsilon}$.

In summary, the problem transformed to the fixed domain is \\
{\bf Problem $\tilde{\rm P}_\epsilon$}. Find $\tueps, \treps: [0, T) \times \bar{\Omega}_\epsilon \rightarrow \R$ solving 
\begin{subequations}\label{MicProbFixedDomain}
\begin{align}
\partial_t (\jeps \tueps) - V_{\epsilon} \cdot \nabla \tueps - \tueps \partial_t \jeps &= \nabla \cdot \big( D_{\epsilon} \nabla \tueps \big) +  \jeps f_{\epsilon} &\mbox{ in }& (0,T)\times \oe,
\\
-D_{\epsilon} \nabla \tueps \cdot \nu_{\epsilon} &= \treps'(\tueps - \rho) \jeps &\mbox{ on }& (0,T)\times \geps,
\\
\tueps &=0 &\mbox{ on }& (0,T)\times \partial \Omega,
\\
\partial_t \treps &= G_{\epsilon}(t,x,\tueps , \treps) &\mbox{ in }& (0,T) \times \oe,
\\
\label{MicProbFixedDomainIVueps}
\tueps(0) &=  \tuepsin &\mbox{ in }& \oe ,
\\
\label{MicProbFixedDomainIVreps}
\treps(0) &= \trepsin  &\mbox{ in }& \oe.
\end{align}
\end{subequations}
The definition of a weak solution of Problem $\tilde{\rm P}_\epsilon$ is given in  Definition \ref{DefMicProbFixedDomain} below.

We now state the assumptions on the data \\[0.5em]
\noindent\textbf{Assumptions on the data:}
\begin{enumerate}[label = (A\arabic*)]
\item\label{AssumptionInitialValueUeps} For the microscopic initial data $\uepsin\in H^1(\oe)$ a $C > 0$ exists such that 
\begin{align*}
\Vert \uepsin \Vert_{H^1(\oe)} \le C 
\end{align*}
uniformly w.r.t $\epsilon$. 
Further there exists $\uin \in H^1(\Omega)$ such that $\uepsin$ converges in the two-scale sense to $\uin$.
\item \label{AssumptionInitialValueReps} The initial radii $\repsin \in L^{\infty}(\Omega)$ are constant on every cell $\epsilon (Y + {\bf k})$, ${\bf k} \in K_{\epsilon}$. Furthermore, there exist $\uR, \oR \in \R$, $0 < \uR < \oR < \frac 1 2$, such that
\begin{align*}
\epsilon \uR \le \repsin \le \epsilon \oR \qquad \text{  in a.e. sense}.
\end{align*}
Also, we assume that a $C >0$ exists such that for any $0 \kl h \ll 1$ and $\Bell \in \Z^n$ with $\vert \epsilon \Bell \vert \kl h$ it holds 
\begin{align*}
\epsilon^{-1} \Vert \repsin(\cdot + \Bell \epsilon) - \repsin \Vert_{L^2(\Omega^h)} \le C \vert \Bell \epsilon \vert,
\end{align*}
with $\Omega^h:= \{x \in \Omega \, : \, \mathrm{dist}(\partial \Omega,x) \gr h \}$. Additionally, there exists $\rin \in L^{\infty}(\Omega)$, such that 
\begin{align*}
\epsilon^{-1}\repsin \rightarrow \rin \quad \mbox{ in the two-scale sense}
\end{align*}
in $L^p$ for every $p \in [1,\infty)$.
\item\label{AssumptionODEg} For the  (precipitation/dissolution) rate one has $g \in C^1(\R^2) \cap   W^{1,\infty}(\R^2)$.
\item\label{AssumptionRbounded} There exists $0<\delta_0 < \oR - \uR$, such that for all $u \in \R$ it holds that
\begin{enumerate}
[label = (\roman*)]
\item $g(u,R) \le 0 $ for all $R \in \big[\oR - \delta_0, \oR\big]$,
\item $g(u,R) \geq 0 $ for all $R \in \big[\uR, \uR + \delta_0\big]$.
\end{enumerate}
\item\label{AssumptionRechteSeitef} The source term is continuous, $f \in C^0\big([0,T]\times \overline{\Omega})$.
\item\label{AssumptionDiffusion} The diffusion tensor $D \in C^0\big(\overline{\Omega}\big)^{n\times n}$ is symmetric and coercive, \ie there exists $c_0>0$ such that 
\begin{align*}
D(x) \xi \cdot \xi \geq c_0 \quad \mbox{ for all } x \in \overline{\Omega}, \, \xi \in \R^n.
\end{align*}
\end{enumerate}

Note that, by the assumption on the shifts of $\repsin$ in Assumption  \ref{AssumptionInitialValueReps},  $\epsilon^{-1} \repsin$ is relatively compact in $L^2(\Omega)$, and  $\rin \in H^1(\Omega)$, see Remark \ref{RemarkKonvergenzReps}.

\subsection{Weak formulation of the micro model}

Now, we give the definition of the weak solution of the microscopic model $\eqref{MicroscopicModel}$. In the following, we suppress the notation $\tilde{\cdot}$, \ie the solution of Problem $\tilde{\rm P}_\epsilon$  
is denoted (by an abuse of notation) by $(\ueps,\reps)$. We consider weak solutions, as defined below.
\begin{definition}\label{DefMicProbFixedDomain}
The pair $(\ueps,\reps)$ with $\ueps  \in L^2((0,T),H^1(\oe))$ such that $\partial_t(\jeps \ueps) \in L^2((0,T),H^1(\oe)')$ 
and $\reps \in W^{1,\infty}((0,T),L^2(\oe))$  s.t. $\reps,\, \partial_t \reps \in L^{\infty}((0,T)\times \oe)$, and $\epsilon \uR \le \reps \le \epsilon \oR$  is a weak solution of Problem $\tilde{\rm P}_\epsilon$ if, for all $\phi \in H^1(\oe)$, it holds for almost every $t\in (0,T)$ that 
\begin{align}
\begin{aligned}\label{WeakFormulationMicroscopicModel}
\langle \partial_t & \big(\jeps \ueps\big) ,\phi \rangle_{H^1(\oe)',H^1(\oe)} - \int_{\oe} \ueps \partial_t \jeps \phi dx - \int_{\oe} V_{\epsilon} \cdot \nabla \ueps \phi dx 
\\
&+ \int_{\oe} D_{\epsilon}\nabla \ueps \cdot \nabla \phi dx = \int_{\oe} \jeps f_{\epsilon} \phi dx  - \int_{\geps} \partial_t \reps (\ueps - \rho) \jeps \phi d\sigma ,
\end{aligned}
\end{align}
and 
\begin{align}\label{DefMicProblFixedDomainODEReps}
\partial_t \reps(t,x) =  \frac{\epsilon}{|\epsilon \Gamma|} \int_{\geps\left(\epsilon \bfxe\right)} g \left( \ueps (t,z),\frac{\reps(t,z)}{\epsilon}\right) d\sigma.
\end{align}
Additionally, $(\ueps, \reps)$ fulfill the initial conditions $\eqref{MicProbFixedDomainIVueps}$ and $\eqref{MicProbFixedDomainIVreps}$.
\end{definition}

\begin{remark}\label{BemerkungRegReps}\
\begin{enumerate}[label = (\roman*)]
    \item We emphasize that due to the mean value theorem, see \cite[Chapter 5.9, Theorem 2]{EvansPartialDifferentialEquations}, it holds that $\reps \in C^{0,1}([0,T], L^{\infty}(\oe))$. Therefore, \eqref{DefMicProblFixedDomainODEReps} holds  for all $t \in [0, T]$ and almost every $x \in \Omega$. 
    \item In the following we extend $\reps$ constantly from  every perforated micro cell $\epsilon (Y^{\ast} + \bf k)$  with ${\bf k} \in K_{\epsilon}$ to the whole cell $\epsilon (Y + \bf k)$, and therefore $\reps$ can be treated as a function  defined on the whole domain $\Omega$.
\end{enumerate}
\end{remark}
We will see that $\partial_t \ueps \in L^2((0,T),H^1(\oe)')$ and  the initial condition for $\ueps$ is well-defined. However, the norm of $\partial_t \ueps$ is of order $\epsilon^{-1}$ and therefore we  work with the time-derivative $\partial_t (\jeps \ueps)$ to establish convergence results for the micro solutions to pass to the limit $\epsilon \to 0$.

\begin{remark}\label{RemarkIdentityJacobianVelocity}
In fact, \eqref{WeakFormulationMicroscopicModel}can be rewritten as 
\begin{equation}\label{eq:tildeu2}
\begin{array}{l}
\langle \partial_t \big(\jeps \ueps\big) ,\phi \rangle_{H^1(\oe)',H^1(\oe)} + \int_{\oe} \ueps \veps \cdot \nabla \peps dx - \int_{\partial \oe} \ueps \peps \veps \cdot \nu d\sigma 
\\[0.5em]
\qquad + \int_{\oe} D_{\epsilon}\nabla \ueps \cdot \nabla \phi dx = \int_{\oe} \jeps f_{\epsilon} \phi dx  - \int_{\geps} \partial_t \reps (\ueps - \rho) \jeps \phi d\sigma ,
\end{array}
\end{equation}
This is a consequence of the equality  
\begin{align*}
- \int_{\oe} \tueps \partial_t \jeps \peps dx - \int_{\oe} \veps \cdot \nabla \tueps \peps dx   = \int_{\oe} \tueps \veps \cdot \nabla \peps dx - \int_{\partial \oe} \tueps \peps \veps \cdot \nu d\sigma, 
\end{align*}
all $\tueps, \peps \in H^1(\oe)$, 
where $\jeps, \veps$ are defined in \eqref{DefFepsJeps}--\eqref{DefVepsDeps}. This holds since  
\begin{align}\label{eq:Piolaconsequence}
\partial_t \jeps = \nabla \cdot \veps 
\end{align} 
(see \cite[p. 105]{GahnNRP}),
which is a direct consequence of Piola's  identity $\nabla \cdot (\jeps \feps^{-1}) = 0$, see e.g. \cite[p. 117]{MarsdenHughes}. 
More precisely, for the term on the left in \eqref{eq:Piolaconsequence} one has 
\begin{align}\label{JacobiFormula}
\partial_t \jeps = \jeps tr\big(\feps^{-1} \partial_t \feps \big) 
\end{align}
due to the Jacobi formula. For the term on the right, the product rule gives  
\begin{align*}
\nabla \cdot \veps =  tr\big(\jeps \feps^{-1} \nabla \partial_t \seps\big) + \big[\nabla \cdot (\jeps \feps^{-1}) \big] \cdot \partial_t \seps . \end{align*}
Using  Piola's identity,  \eqref{eq:Piolaconsequence} follows immediately. 
\end{remark}

\section{Main results}
The aim of the paper is two-folded.  First, we show the existence of a  weak solution of the microscopic problem in $\eqref{MicProbFixedDomain}$ together with uniform \textit{a priori} estimates with respect to $\epsilon$. The following theorem is proven in Section \ref{SectionExistence}. 
\begin{theorem}\label{ExistenceTheorem}
There exists a weak solution of the (transformed, microscopic) Problem $\tilde{\rm P}_\epsilon$  
in the sense of Definition \ref{DefMicProbFixedDomain}. This solution fulfills the \textit{a priori} estimates in Lemma \ref{AprioriEstimatesSummary}.
\end{theorem}
In a second step we use the $\epsilon$-uniform estimates to  derive two-scale compactness results for the micro solutions, and show that the $\epsilon \to 0$ limit functions 
$u_0: (0,T) \times \Omega \rightarrow \R$ and $R_0: (0,T) \times \Omega \rightarrow \R$
solve a macroscopic model with homogenized coefficients. 
In this sense, we start by defining 
\begin{align*}
    S_0(t,x,y) &= y + \big(R_0(t,x) - \oR\big) (\chi_0 \nu_0) (y),
\end{align*}
and 
\begin{align*}
    J_0(t,x,y) &= \det(\nabla_y S_0(t,x,y)),
    \\
    V_0(t,x,y) &= J_0(t,x,y) \nabla_y S_0(t,x,y)^{-1} \partial_t S_0(t,x,y),
\end{align*}
and the averaged quantities 
\begin{align*}
    q(t,x) &:= \int_{Y^{\ast}} \nabla_y \cdot V_0(t,x,y) dy,
    \\
    \bar{J}_0(t,x) &:= \int_{Y^{\ast}} J_0(t,x,y) dy.
\end{align*}
Further we define the moving cell surface as 
\begin{align*}
\Gamma(t,x):= \partial B_{R_0(t,x)}(x).
\end{align*}
Now, the macroscopic model is \\
{\bf Problem ${\rm P}_0$}. Find $u_0, R_0: [0, T) \times \bar{\Omega} \rightarrow \R$ solving 
\begin{align}
\begin{aligned}\label{MacroModelStrong}
\partial_t \big(\bar{J}_0 u_0\big) &- u_0 q - \nabla \cdot \big(D_0^{\ast} \nabla u_0 \big) \\
& = \int_{Y^{\ast}} J_0 f_0 dy - \vert \Gamma (t,x)\vert \partial_t R_0 (u_0 - \rho) &\mbox{ in }& (0,T)\times \Omega,
\\
-D_0^{\ast} \nabla u_0 \cdot \nu &= 0 &\mbox{ on }& (0,T)\times \partial \Omega,
\\
\partial_t R_0 &= g(u_0,R_0) &\mbox{ in }& (0,T)\times \Omega,
\\
u_0(0) &= \uin &\mbox{ in }& \Omega,
\\
 R_0(0) &= \rin &\mbox{ in }& \Omega,
\end{aligned}
\end{align}
where $D_0^{\ast}$ is the homogenized diffusion coefficient defined in $\eqref{DefinitionHomogenizedDiffusion}$ via the cell problems $\eqref{CellProblems}$.

We proceed with the definition of a weak solution for the macroscopic model.
\begin{definition}\label{def:weakMacro}
A weak solution of the (macroscopic) Problem ${\rm P}_0$ 
is a pair $(u_0,R_0)$ satisfying 
\begin{align*}
    u_0 &\in L^2((0,T),H^1(\Omega)) \quad \mbox{with } \partial_t \big(\bar{J}_0 u_0 \big) \in L^2((0,T),H^1(\Omega)'),
    \\
    R_0 &\in W^{1,\infty}((0,T),L^{\infty}(\Omega)),
\end{align*}
and for all $\phi \in H^1(\Omega)$ and almost every $t \in (0,T)$ it holds that
\begin{align}
\begin{aligned}\label{MacroModelVar}
\langle \partial_t (\bar{J}_0 u_0 ) ,& \phi \rangle_{H^1(\Omega)',H^1(\Omega)}  -  \int_{\Omega} q u_0 \phi dx  +   \int_{\Omega} D_0^{\ast} \nabla u_0 \cdot \nabla \phi dx  
\\
&=   \int_{\Omega} \int_{Y^{\ast}} J_0 f_0 dy \phi dx  -  \int_{\Omega} \partial_t R_0 (u_0 - \rho) \phi \vert \Gamma(t,x)\vert dx .
\end{aligned}
\end{align}
Additionally, it holds almost everywhere in $(0,T)\times \Omega$
\begin{align*}
    \partial_t R_0 = g(u_0,R_0).
\end{align*}
Further, there hold the initial conditions $(\bar{J}_0u_0)(0) = \bar{J}_0(0) \uin$ and $R_0(0) = \rin$.
\end{definition}
 We emphasize the from the function spaces in the Definition \ref{def:weakMacro} and the fact that $\bar{J}_0 \in L^{\infty}((0,T),L^{\infty}(\Omega)) \cap L^{\infty}((0,T),H^1(\Omega))$, see Remark \ref{RegularityBarJ0} and Remark \ref{RemarkMainResult}, it follows immediately that $(\bar{J}_0u_0) \in C^0([0,T],L^2(\Omega))$, whereas it is not obvious in which sense we can understand $u_0(0)$, see also Remark \ref{RemarkMainResult}.
In the remaining part of this work, we prove that the macroscopic (weak) solution defined above is obtained as the limit $\epsilon \to 0$ of solutions to the microscopic problems in \eqref{MicProbFixedDomain}. 
\begin{theorem}\label{MainResultConvergenceMacroModel}
Let $(\ueps,\reps)$ be a weak solution of the (microscopic) Problem $\tilde{\rm P}_\epsilon$,  
in the sense of Definition \ref{DefMicProbFixedDomain}. Then, up to a sequence $\epsilon \to 0$, it holds that 
\begin{align}
    \ueps &\rightarrow u_0 &\mbox{ strongly in the two-scale sense in }& L^2,
    \\
    \epsilon^{-1} \reps &\rightarrow R_0 &\mbox{ in }& L^p((0,T)\times \Omega)
\end{align}
for all $p \in [1,\infty)$, where the pair $(u_0,R_0)$ is a weak solution of the macroscopic model $\eqref{MacroModelStrong}$, in the sense of Definition \ref{def:weakMacro}. Additionally, we have $u_0 \in C^0([0,T],L^2(\Omega))$ and $u_0 (0) = u_0$.
\end{theorem}
For the definition of the two-scale convergence we refer to the Appendix \ref{SectionTwoScaleConvergence}. The proof of Theorem \ref{MainResultConvergenceMacroModel} is given in Section \ref{SectionDerivationMacroModel}.

\section{Existence of a microscopic solution and \textit{a priori} estimates}
\label{SectionExistence}

In this section we  use the Rothe-method to establish the existence of a weak solution of the microscopic model $\eqref{MicProbFixedDomain}$,  and therefore $\eqref{DefSolutionMicProblem}$, in the sense of Definition \ref{DefMicProbFixedDomain}. Futher, we show \textit{a priori} estimates for the solution depending explicitly on the parameter $\epsilon$. These \textit{a priori} estimates form the basis for the derivation of the macroscopic model.

\subsection{The time-discretized model}

We start by formulating the time-discretized model. For $N\in \N$ we define the time step $\dt := \frac{T}{N} $ and set $t^j:= j \dt $ for $j\in \{0,\ldots,N\}$. Then, we are looking for a sequence of solutions $(\ujeps,\rjeps)$ of the time-discretized problems. Clearly, for $j=0$ we take $\ueps^0 = \uepsin$ and $\reps^0 = \repsin$. 

\begin{definition}
\label{DefTimeDiscModel}
Let $j\in \{1,\ldots,N\}$ and $(\ueps^{j-1},\reps^{j-1}) \in H^1(\oe)\times L^{\infty}(\Roe)$ be given. A pair  $(\ujeps,\rjeps) \in H^1(\oe)\times L^{\infty}(\Roe)$ is a solution of the time-discrete problem at time $t^j$ if 
\begin{align}\label{VarEquTimeDiscModelR}
\frac{\rjeps - \reps^{j-1}}{\dt} = G_{\epsilon}^{j-1} \quad \mbox{ in } \Roe,
\end{align}
and for all $\phi \in H^1(\oe)$ it holds that
\begin{align}\label{VarEquTimeDiscModel}
\begin{aligned}
\int_{\oe} & \frac{\jjeps \ujeps - \jjoeps \ujoeps}{\dt} \phi dx + \int_{\oe} D_{\epsilon}^j \nabla \ujeps \cdot \nabla \phi dx
\\
=& \int_{\oe} V_{\epsilon}^j \cdot \nabla \ujeps \phi dx + \int_{\oe} \jjeps f_{\epsilon}^j \phi dx + \int_{\oe} \ujeps \frac{\jjeps - \jjoeps}{\dt} \phi dx
\\
&- \int_{\geps} G_{\epsilon}^{j-1} (\ujoeps - \rho ) \phi \jjeps d\sigma .
\end{aligned}
\end{align}
The time-discrete coefficients above are defined as 
\begin{align}\label{DefVepsDeps_j}
\begin{aligned}
G_{\epsilon}^{j-1}(x) &:= \frac{\epsilon}{|\epsilon \Gamma|} \int_{\geps\left(\epsilon \bfxe\right)} g \left(\ujoeps,\frac{\rjoeps}{\epsilon}\right) d\sigma ,
\\
S_{\epsilon}^j(x)&:= x + (\rjeps - \epsilon \oR ) \chi_0 \left(\fxe\right) \nu_0 \left(\fxe\right),
\\
F^j_{\epsilon}(x)&:=  I_n + \frac{\rjeps - \epsilon \oR}{\epsilon} \left[\nabla_y  \left(\chi_0 \nu_0\right)\right] \left(\fxe\right),
\\
\jjeps(x) &:= \det F^j_{\epsilon}(x),
\\
D_{\epsilon}^j(x) &:= \jjeps(x) \left[F_{\epsilon}^j(x)\right]^{-1} D\big(S_{\epsilon}^j(x)\big)  \left[F_{\epsilon}^j(x)\right]^{-T},
\\
V_{\epsilon}^j(x) &:= \jjeps(x)  \left[F_{\epsilon}^j(x)\right]^{-T} \frac{\rjeps(x) - \rjoeps(x)}{\dt}\chi_0 \left(\fxe\right) \nu_0 \left(\fxe\right),
\\
f_{\epsilon}^j(x) &:= f_{\epsilon}(t^j,x).
\end{aligned}
\end{align}
\end{definition}
\begin{remark}\label{rem:Linear_TimeDiscrete}
Given $\ujoeps, \rjoeps$, the existence of $\rjeps$ follows straightforwardly from \eqref{VarEquTimeDiscModelR}. Moreover, this also gives $\jjeps$, see \eqref{DefVepsDeps_j}, which means that, in fact, \eqref{VarEquTimeDiscModel} providing $\ujeps$ is linear. Moreover, as follows from Proposition \ref{PropEstimateDiscTimeReps} below,  $\rjeps$ is constant on every microcell $\epsilon(Y + {\bf k})$, assuming that the same holds for $\rjoeps$. Therefore we have $F^j_{\epsilon} = \nabla S_{\epsilon}^j$ in classical sense.
\end{remark}
The constant $C>0$ used below is generic and does not depend on $\dt $ or $\epsilon$.
\begin{proposition}\label{PropEstimateDiscTimeReps}
Let $j \in \{1,\ldots,N\}$ and $(\ujoeps, \rjoeps) \in H^1(\oe) \times L^{\infty}(\Roe)$ with $ \epsilon \uR \le \rjoeps \le \epsilon \oR $, and assume $\rjoeps$ is constant on every microscopic cell $\epsilon (Y + {\bf k})$ with ${\bf k} \in K_{\epsilon}$. Then, for $\dt$ small enough, we have $\rjeps \in L^{\infty}(\Roe)$ with $\epsilon \uR \le \rjeps \le \epsilon \oR$. Also, $\rjeps $ is constant on $\epsilon (Y + {\bf k})$, and one has 
\begin{align}\label{EstimateDiscreteTimeDerivReps}
\left\| \frac{\rjeps - \rjoeps}{\dt}\right\|_{L^{\infty}(\Roe)} \le C  \epsilon.
\end{align}
\end{proposition}

\begin{proof}
From the assumption $g \in L^{\infty}(\R \times \R)$ in \ref{AssumptionODEg} one obtains that $\rjeps \in L^{\infty}(\Roe)$, and that $\eqref{EstimateDiscreteTimeDerivReps}$ holds true. Further, the definition of $\rjeps$ implies that the function is constant on every $ \epsilon (Y + {\bf k})$ with ${\bf k}\in K_{\epsilon}$. It remains to check the lower and upper bound of $\rjeps$: We only prove the upper bound, since the lower bound follows by the same arguments. For almost every $x \in \Roe$, if $\rjoeps (x)  \in [\epsilon (\oR - \delta_0) , \epsilon \oR\big]$ we use Assumption \ref{AssumptionRbounded} to obtain  
\begin{align*}
\rjeps (x) \le \rjoeps(x) \le \epsilon \oR.
\end{align*} 
If $\epsilon \uR \le \rjoeps(x)  < \epsilon(\oR - \delta_0)$, taking $\dt \le \delta_0/\|g\|_{L^{\infty}(\R \times \R)}$ gives 
\begin{align*}
\rjeps = \rjoeps + \dt G_{\epsilon}^{j-1} \le \rjoeps + \epsilon \dt \|g\|_{L^{\infty}(\R \times \R)} \le \rjoeps + \epsilon \delta_0 < \epsilon \oR.
\end{align*}
Finally, \eqref{EstimateDiscreteTimeDerivReps} is a direct consequence of Assumption \ref{AssumptionODEg}.  
\end{proof}

The boundedness of $\rjeps$ gives uniform estimates for the time-discrete coefficients, as follows from the lemma below. 
\begin{lemma}\label{LemmaEstimatesTimeDisCoeff}
Let $j \in \{1,\ldots,N\}$, and $\ujoeps \in H^1(\oe)$, as well as $\rjoeps, \rjeps \in L^{\infty}(\Roe)$ be given. Assume that $ \epsilon \uR \le \rjoeps, \rjeps \le \epsilon \oR $ and $\rjoeps, \rjeps $ are constant in every  cell $\epsilon (Y + {\bf k})$ with ${\bf k} \in K_{\epsilon}$.
Then we have $\seps^j, F_{\epsilon}^j \in C^{\infty}(\overline{\oe})^n$, satisfying 
\begin{subequations}\label{EstimatesTimeDiscCoeff}
\begin{align}\label{EstimatesTimeDiscCoeffSeps}
\|\seps^j\|_{C^1(\overline{\oe})} = \|\seps^j\|_{C^0(\overline{\oe})} + \|F_{\epsilon}^j\|_{C^0(\overline{\oe})} \le C.
\end{align}
Similarly, for the Jacobi-determinant we have $\jjeps \in C^{\infty}(\overline{\oe})$, with 
\begin{align}\label{EstimatesTimeDiscCoeffJ}
\left(\frac{\uR}{\oR}\right)^{n-1} \le\jjeps \le C.
\end{align}
Further, we have $D_{\epsilon}^j \in C^0(\overline{\oe})^{n\times n}$ and $V_{\epsilon}^j \in C^0(\overline{\oe})^n$, with
\begin{align}\label{EstimatesTimeDiscCoeffDV}
 \|D_{\epsilon}^j\|_{C^0(\overline{\oe})} + \|V_{\epsilon}^j\|_{L^{\infty}(\oe)} \le C.
\end{align}
Additionally, it holds that
\begin{align}\label{EstimatesTimeDiscCoeffDtJ}
\left\| \frac{\jjeps - \jjoeps}{\dt}\right\|_{C^0(\overline{\oe})} \le C.
\end{align}
\end{subequations}
\end{lemma}
\begin{proof}
First, we observe that, since $F_{\epsilon}^j = \nabla \seps^j \in C^{\infty}(\overline{\oe})^{n\times n}$, the regularity and the estimates for $F_{\epsilon}^j$ are a direct consequence of the results for $\seps^j$. 

Next, since $\rjeps$ is constant on $\epsilon (Y + {\bf k})$ for every ${\bf k} \in K_{\epsilon}$, $\chi_0$ and $\nu_0$ are smooth, and $\chi_0\left(\frac{\cdot}{\epsilon}\right)$ has compact support in $\epsilon (Y + {\bf k})$. This immediately implies  the smoothness of $\seps^j$, $F_{\epsilon}^j$, $\jjeps$, and $V_{\epsilon}^j$.  Further, we have $F_{\epsilon}^j = E_n$ on the lateral boundary of the micro cell $\partial ( \epsilon (Y + {\bf k}))$, 
and, Assumption \ref{AssumptionDiffusion} gives $D_{\epsilon}^j \in C^0(\overline{\oe})^{n\times n}$.
%
%
%
%
Inequality $\eqref{EstimatesTimeDiscCoeffSeps}$ follows immediately from  $\epsilon \uR \le \rjeps \le \epsilon \oR$. For the estimate $\eqref{EstimatesTimeDiscCoeffJ}$, we first notice that the regularity of the determinant and the essential boundedness of $F_{\epsilon}^j$ implies $\|\jjeps\|_{L^{\infty}(\oe)} \le C$. The lower bound is obtained from the results in Appendix 
\ref{SectionAppendixHanzawaReferenceElement}.

To prove $\eqref{EstimatesTimeDiscCoeffDV}$, one needs to control the inverse of $F_{\epsilon}^j(x)$. This is achieved by using the following inequality, stating that a $C > 0$ exists such that, for any invertible matrix $A \in \R^{n\times n}$, and for an arbitrary norm in $\R^{n\times n}$, 
\begin{align}\label{BoundInverseMatrix}
\|A^{-1}\| \le \frac{C}{|\det (A)|} \|A\|^{n-1}.
\end{align}
This inequality is straightforward for the spectral norm in $\R^{n\times n}$, and extends to any norm by the equivalence of norms in finite dimensional vector spaces. Then, by \eqref{EstimatesTimeDiscCoeffJ}, for almost every $x \in \oe$ one gets  
\begin{align*}
\left\| \left[F_{\epsilon}^j(x)\right]^{-1}\right\| \le \frac{C}{\jjeps(x)} \left\|F_{\epsilon}^j(x) \right\|^{n-1} \le C , 
\end{align*}
implying $\eqref{EstimatesTimeDiscCoeffDV}$.  Finally,  $\eqref{EstimatesTimeDiscCoeffDtJ}$ follows from the local Lipschitz continuity of the determinant, the essential boundedness of $F_{\epsilon}^j$, and the inequality $\eqref{EstimateDiscreteTimeDerivReps}$.
\end{proof}

We are now able to prove the existence of a weak solution for the time-discrete problems,  
together with some $\epsilon$-dependent \textit{a priori} estimates.
\begin{proposition}\label{ExistenceTimDisProb}
Let $j \in \{1,\ldots,N\}$ and $(\ujoeps, \rjoeps) \in H^1(\oe) \times L^{\infty}(\Omega)$ be given such that $ \epsilon \uR \le \rjoeps \le \epsilon \oR $, and $\rjoeps $ is constant on every  cell $\epsilon (Y + {\bf k})$, ${\bf k} \in K_{\epsilon}$. For $\dt$ small enough, there exists a  weak solution $(\ujeps,\rjeps)$ of the time-discrete problem from Definition \ref{DefTimeDiscModel}. 
\end{proposition}
\begin{proof}
As mentioned in Remark \ref{rem:Linear_TimeDiscrete}, the existence of $\rjeps$ is straightforward, while $\ujeps$ is the weak solution of a linear problem. Moreover, from the proof of Proposition \ref{PropEstimateDiscTimeReps}, one gets that choosing $\dt \leq \delta_0/\|g\|_{L^{\infty}(\R \times \R)}$ guarantees that the estimates there, and therefore in Lemma \ref{LemmaEstimatesTimeDisCoeff} as well, are valid. The variational equation $\eqref{VarEquTimeDiscModel}$ can be rewritten as 
\begin{align*}
a_{\epsilon}^j(\ujeps,\phi) = l_{\epsilon}^j(\phi) \qquad \mbox{for all } \phi \in H^1(\oe),
\end{align*}
with $a_{\epsilon}^j: H^1(\oe)\times H^1(\oe) \rightarrow \R$ and $l_{\epsilon}^j : H^1(\oe) \rightarrow \R$ defined by 
\begin{align*}
a_{\epsilon}^j(u,\phi)&:= \int_{\oe} \frac{{\jeps^{j-1}}}{\dt} u \phi dx + \int_{\oe} D_{\epsilon}^j\nabla u \cdot \nabla \phi dx - \int_{\oe} V_{\epsilon}^j\cdot \nabla u \phi dx,
\\
l_{\epsilon}^j(\phi)&:= \int_{\oe} \frac{\jjoeps}{\dt} \ujoeps \phi dx + \int_{\oe} \jjeps f_{\epsilon}^j \phi dx - \int_{\geps} G_{\epsilon}^{j-1} \big(\ujoeps - \rho\big) \jjeps \phi d\sigma.
\end{align*}
Obviously, $a_{\epsilon}^j$ is bilinear and $l_{\epsilon}^j$ is linear, and both are continuous. To apply the Lax-Milgram Lemma we have to prove the coercivity of the bilinear-form $a_{\epsilon}^j$. Since $\jeps^{j-1}$ is bounded away from 0 (uniformly w.r.t. $\epsilon$ and $j$) and $D$ is coercive, we get the coercivity of $D_{\epsilon}^j$, \ie there exists a constant $d>0$, such that for every $x\in \oe$ it holds that
\begin{align*}
D_{\epsilon}^j(x) \xi \cdot \xi \geq d \quad \mbox{ for all } \xi \in \R^n.
\end{align*} 
Further, the essential boundedness of $V_{\epsilon}^j$ implies the  existence of a constant $C_1>0$, such that
\begin{align*}
-\int_{\oe} u \ V_{\epsilon}^j \cdot \nabla u dx \geq - C_1 \|u \|_{L^2(\oe)}^2 - \frac{d}{2} \|\nabla u\|^2_{L^2(\oe)}.
\end{align*}
Together, we obtain for all $u\in H^1(\oe)$
\begin{align*}
a_{\epsilon}^j(u,u) \geq \int_{\oe} \left[\frac{\jeps^{j-1}}{\dt} - C_1\right] |u|^2 dx + \frac{d}{2} \|\nabla u \|^2_{L^2(\oe)}.
\end{align*}
Recalling the lower bounds for $\jeps^{j-1}$, which are uniform with respect to $j$, this implies the coercivity of $a_{\epsilon}^j$ for $\dt$ small enough, and the claim is proved.
\end{proof}

\begin{proposition}\label{AprioriEstimateDiscreteSolutionUeps}
Assume $\dt$ small enough. Then, a $C > 0$ not depending on of $\epsilon $ and $\dt$ exists such that the sequence of time-discrete solutions $(\ujeps,\rjeps)$ ($j=1,\ldots,N$) satisfies the estimates  
\begin{align*}
\max_{j\in \{1,\ldots,N\}} \|\ujeps\|_{L^2(\oe)} + \dt \sum_{j=1}^N \|\nabla \ujeps\|^2_{L^2(\oe)} \le C . 
\end{align*}
\end{proposition}
\begin{proof}
We take $\dt \ujeps $ as test function in $\eqref{VarEquTimeDiscModel}$ to obtain
\begin{align*}
A_{\epsilon}^1 + A_{\epsilon}^2:&= \int_{\oe} \big(\jjeps \ujeps  - \jjoeps \ujoeps \big) \ujeps dx + \dt \int_{\oe} D_{\epsilon}^j \nabla \ujeps \cdot \nabla \ujeps dx
\\
&= \dt \int_{\oe} V_{\epsilon}^j \cdot \nabla \ujeps  \ujeps dx + \dt \int_{\oe} \jjeps f_{\epsilon}^j \ujeps dx +\dt \int_{\oe} \big(\ujeps\big)^2 \frac{\jjeps - \jjoeps}{\dt} dx 
\\
& \hspace{2em} - \dt\int_{\geps} G_{\epsilon}^{j-1} \big(\ujoeps - \rho \big) \ujeps \jjeps d\sigma
\\
&=: \sum_{i=1}^4 B_{\epsilon}^i.
\end{align*}
For the first term on the left-hand side we have
\begin{align*}
A_{\epsilon}^1=& \frac12 \big\|\sqrt{\jjeps} \ujeps \big\|^2_{L^2(\oe)} + \frac12 \big\|\sqrt{\jjeps } \ujeps - \sqrt{\jjoeps} \ujoeps \big\|^2_{L^2(\oe)} 
\\
&- \frac12 \big\|\sqrt{\jjoeps} \ujoeps \big\|^2_{L^2(\oe)}
 - \int_{\oe} \sqrt{\jjoeps} \big( \sqrt{\jjoeps} - \sqrt{\jjeps}\big) \ujeps \ujoeps dx.
\end{align*}
We denote the last term above by $B_{\epsilon}^5$. Now, using the coercivity of $D_{\epsilon}^j$ as in the proof of Proposition \ref{ExistenceTimDisProb}, we obtain
\begin{align*}
A_{\epsilon}^2 \geq \dt \ d \ \big\| \nabla \ujeps\big\|^2_{L^2(\oe)}.
\end{align*}
To estimate the terms $B_{\epsilon}^i$ ($i = 1, \dots, 5$), we first observe that 
\begin{align*}
|B_{\epsilon}^1| \le \|V_{\epsilon}^j\|_{L^{\infty}(\oe)} \dt \|\nabla \ujeps\|_{L^2(\oe)} \|\ujeps\|_{L^2(\oe)} \le C\dt \|\ujeps\|^2_{L^2(\oe)} + \frac{d \dt}{4} \|\nabla \ujeps\|^2_{L^2(\oe)}.
\end{align*}
For the second and the third term, we use  Lemma \ref{LemmaEstimatesTimeDisCoeff} to obtain 
\begin{align*}
|B_{\epsilon}^2| &\le C\dt \big( \|f_{\epsilon}^j\|_{L^2(\oe)}^2 + \|\ujeps \|^2_{L^2(\oe)} \big),
\\
|B_{\epsilon}^3| &\le C \dt \|\ujeps \|^2_{L^2(\oe)} .
\end{align*}
For $B_{\epsilon}^4$ we use that $\|G_{\epsilon}^{j-1}\|_{L^{\infty}(\geps)} \le C \epsilon$, and the scaled trace inequality, 
\begin{align}\label{ScaledTraceInequality}
\epsilon \|\ueps \|_{L^2(\geps)}^2 \le C \|\ueps\|^2_{L^2(\oe)} + \frac{d}{4} \epsilon^2 \|\nabla \ueps \|^2_{L^2(\oe)} 
\end{align}
for all $\ueps \in H^1(\oe)$, which can be obtained by a standard decomposition argument for $\oe$. Then, for $\epsilon \le 1$, 
\begin{align*}
|B_{\epsilon}^4| &\le C  \dt \left(1 +  \epsilon\|\ujeps\|^2_{L^2(\geps)}+ \epsilon\|\ujoeps\|^2_{L^2(\geps)}\right)
\\
&\le C \dt \left( 1 + \|\ujeps\|^2_{L^2(\oe)} + \|\ujoeps\|^2_{L^2(\oe)} \right) + \frac{d \dt}{4} \left( \|\nabla \ujeps\|^2_{L^2(\oe)}  + \|\nabla \ujoeps\|^2_{L^2(\oe)} \right).
\end{align*}
For the last term $B_{\epsilon}^5$ occuring in the term $A_{\epsilon}^1$, we use the Lipschitz continuity of $\sqrt{\cdot}$ away from $0$, to obtain with $\eqref{EstimatesTimeDiscCoeffDtJ}$
\begin{align*}
|B_{\epsilon}^5| &\le C \dt \int_{\oe} \frac{\left|\sqrt{\jjeps} - \sqrt{\jjoeps}\right|}{\dt} |\ujeps| | \ujoeps| dx 
\\
&\le C \dt \int_{\oe} \frac{\left|\jjeps - \jjoeps \right|}{\dt} \cdot \big( |\ujeps|^2 + |\ujoeps|^2 \big) dx 
\\
&\le C\dt \left(\|\ujeps\|^2_{L^2(\oe)} + \|\ujoeps\|^2_{L^2(\oe)}\right).
\end{align*}
Altogether, we obtain
\begin{align*}
\frac12 &\big\|\sqrt{\jjeps} \ujeps \big\|^2_{L^2(\oe)} + \frac12 \big\|\sqrt{\jjeps } \ujeps - \sqrt{\jjoeps} \ujoeps \big\|^2_{L^2(\oe)}
\\
 -& \frac12 \big\|\sqrt{\jjoeps} \ujoeps \big\|^2_{L^2(\oe)} + \frac{d \dt}{2} \|\nabla \ujeps \|^2_{L^2(\oe)}
\\
&\le C \dt \left( 1 + \|f_{\epsilon}^j\|^2_{L^2(\oe)} + \|\ujoeps\|^2_{L^2(\oe)} + \|\ujeps\|^2_{L^2(\oe)}\right)  + \frac{d \dt }{4} \Vert \nabla \ujoeps \Vert^2_{L^2(\oe)}.
\end{align*}
Summing up these inequalities over $j=1$ to an arbitrary $k \in \{1,\ldots,N\}$, we use again the boundedness of $\jjeps$ proved above and Assumption \ref{AssumptionInitialValueUeps} to obtain that, for $\dt $ small enough, 
\begin{align*}
\|\ueps^k &\|_{L^2(\oe)}^2 + \sum_{j=1}^k  \big\|\sqrt{\jjeps } \ujeps - \sqrt{\jjoeps} \ujoeps \big\|^2_{L^2(\oe)} + \dt \sum_{j=1}^k \|\nabla \ujeps\|^2_{L^2(\oe)} 
\\
\le& C\left(1 +  \|\uepsin\|^2_{L^2(\oe)} + \dt \sum_{j=1}^k \|f_{\epsilon}^j\|^2_{L^2(\oe)} + \dt \sum_{j=1}^{k} \|\ujeps\|^2_{L^2(\oe)} + \dt \Vert \nabla \uepsin \Vert^2_{L^2(\oe)} \right)
\\
\le& C\left(1 + \|f\|_{C^0(\overline{\Omega},L^2(\Omega))} + \dt \sum_{j=1}^{k-1} \|\ujeps\|^2_{L^2(\oe)}\right).
\end{align*}
Now, the discrete Gronwall-inequality implies the desired result.
\end{proof}

\subsection{The interpolation in time}

In this section, we define the piecewise constant and piecewise linear interpolation of the discrete values with respect to time. 
Let  $(\ujeps,\rjeps)$ for $j \in \{1,\ldots,N\}$ be a weak solution of the time-discrete problem from Proposition \ref{ExistenceTimDisProb} with initial condition $(\uepsin,\repsin)$. For $t \in (t^{j-1},t^j]$ with $j =1,\ldots,N$, we define:
\begin{align*}
\ubteps(t)&:= \ujeps, &\uhteps(t)&:= \ujoeps + \frac{t - t^{j-1}}{\dt} \left( \ujeps - \ujoeps \right),
\\
\rbteps(t) &:= \rjeps , &\rhteps(t)&:= \rjoeps + \frac{t - t^{j-1}}{\dt} \left( \rjeps - \rjoeps\right),
\\
\jbteps(t) &:= \jjeps, &\jhteps(t)&:= \jjoeps + \frac{t - t^{j-1}}{\dt}\left(\jjeps - \jjoeps\right),
\\
\wbteps(t) &:= \jjeps \ujeps, &\whteps(t) &:= \jjoeps \ujoeps + \frac{t - t^{j-1}}{\dt}\left( \jjeps \ujeps - \jjoeps \ujoeps \right).
\end{align*}
For $t<0$ we extend the functions constantly by the initial value.
We emphasize that for the existence  proof of the continuous problem $\eqref{MicProbFixedDomain}$ it is not necessary to consider the interpolations $\wbteps$ and $\whteps$. However, we use these functions to show the boundedness of $\partial_t (\jeps \ueps)$ in $L^2((0,T),H^1(\oe)')$ uniformly with respect to $\epsilon$, what is obtained directly from the \textit{a priori} estimates of $\partial_t \whteps$. We are not able to prove such a uniform bound for $\partial_t \ueps$. This is caused by the fact that the gradient of $\jjeps$ is of order $\epsilon^{-1}$, \ie we have
\begin{align}\label{EstimateGradJeps}
\|\nabla \jjeps\|_{L^{\infty}(\oe)} \le C\epsilon^{-1} \quad \mbox{ for all } j \in \{1,\ldots,N\}.
\end{align}
This can be easily seen from the definition of $\jjeps$ and the results from Proposition \ref{PropEstimateDiscTimeReps}. Let us start with the \textit{a priori} estimates for $\wbteps$ and $\whteps$:

\begin{lemma}\label{AprioriEstimatesWeps}
The functions $\wbteps$ and $\whteps$ fulfill the following \textit{a priori} estimates:
\begin{align*}
\big\|\wbteps\big\|_{L^{\infty}((0,T),L^2(\oe))} + \epsilon \big\|\nabla \wbteps\big\|_{L^2((0,T),L^2(\oe))} &\le C,
\\
\big\|\whteps \big\|_{L^{\infty}((0,T),L^2(\oe))} + \epsilon \big\|\nabla \whteps \big\|_{L^2((0,T),L^2(\oe))} + \big\|\partial_t \whteps \big\|_{L^2((0,T),H^1(\oe)')} &\le C.
\end{align*}
\end{lemma}
\begin{proof}
Except for the time derivative, all estimates follow directly from  Lemma \ref{LemmaEstimatesTimeDisCoeff}, Proposition \ref{AprioriEstimateDiscreteSolutionUeps}, and $\eqref{EstimateGradJeps}$.  Hence, we only give a detailed proof for the inequality of $\partial_t  \whteps$.
 Choose $\phi \in H^1(\oe)$ with $\|\phi\|_{H^1(\oe)} \le 1$ as a test function in $\eqref{VarEquTimeDiscModel}$, to obtain for $t\in (t^{j-1},t^j)$
\begin{align}\label{StartingEquationTimeDerWeps}
\begin{aligned}
\int_{\oe} \partial_t \whteps \phi dx =& - \int_{\oe} D_{\epsilon}^j \nabla \ujeps \cdot \nabla \phi dx + \int_{\oe} V_{\epsilon} \cdot \nabla \ujeps \phi dx + \int_{\oe} \jjeps f_{\epsilon}^j \phi dx 
\\
& +\int_{\oe} \ujeps \frac{\jjeps - \jjoeps}{\dt} \phi dx - \int_{\geps} G_{\epsilon}^{j-1}(\ujoeps - \rho ) \jjeps \phi d\sigma.
\end{aligned}
\end{align}
We only  consider in more detail the boundary term. From the trace inequality $\eqref{ScaledTraceInequality}$ we obtain
\begin{align*}
\|\phi\|_{L^2(\geps)} \le C \left(\frac{1}{\sqrt{\epsilon} } \|\phi \|_{L^2(\oe)} + \sqrt{\epsilon} \|\nabla \phi \|_{L^2(\oe)} \right) \le C \epsilon^{-\frac12}.
\end{align*}
Using again the trace inequality, the estimates $|\geps|\le C\epsilon^{-1}$ and $\Vert G_{\epsilon}^{j-1}\Vert_{L^{\infty}(\geps)} \le C\epsilon$, and Lemma \ref{LemmaEstimatesTimeDisCoeff},  we get
\begin{align*}
- \int_{\geps} G_{\epsilon}^{j-1}(\ujoeps - \rho ) \jjeps \phi d\sigma &\le C\epsilon \left(\|\ujoeps\|_{L^2(\oe)} \|\phi \|_{L^2(\geps)} + \epsilon^{-\frac12} \|\phi\|_{L^2(\geps)} \right)
\\
&\le C \left( \|\ujoeps\|_{L^2(\oe)} + \|\nabla \ujoeps\|_{L^2(\oe)} + 1 \right).
\end{align*}
Hence, from equation $\eqref{StartingEquationTimeDerWeps} $ we obtain using again Lemma \ref{LemmaEstimatesTimeDisCoeff} 
\begin{align*}
\big\|\partial_t \whteps \big\|_{H^1(\oe)'} \le C\left( 1 + \|\ujoeps\|_{L^2(\oe)} + \|\ujeps\|_{L^2(\oe)} + \|\nabla \ujeps \|_{L^2(\oe)} + \|f_{\epsilon}^j\|_{L^2(\oe)} \right).
\end{align*}
The \textit{a priori} estimates from Proposition \ref{AprioriEstimateDiscreteSolutionUeps} and the assumption on $f$ and the initial conditions imply
\begin{align*}
\big\|\partial_t \whteps\big\|^2_{L^2((0,T),H^1(\oe)')} \le C \dt\sum_{j=1}^N \left( 1  + \|\ujeps\|^2_{L^2(\oe)} + \|\nabla \ueps^{j-1} \|^2_{L^2(\oe)} + \|f_{\epsilon}^j\|^2_{L^2(\oe)} \right) \le C,
\end{align*}
what gives us the desired result.
\end{proof}

\begin{lemma}\label{AprioriEstimatesJepsDt}
The functions $\jbteps $ and $\jhteps$ fulfill the following \textit{a priori} estimates:
\begin{align*}
\big\|\jbteps \big\|_{L^{\infty}((0,T),C^0(\overline{\oe}))} + \epsilon \big\|\nabla \jbteps \big\|_{L^{\infty}((0,T),C^0(\overline{\oe}))} &\le C,
\\
\big\|\jhteps \big\|_{W^{1,\infty}((0,T),C^0(\overline{\oe}))} + \epsilon \big\|\nabla \jhteps \big\|_{C^0([0,T] \times \overline{\oe})}   &\le C.
\end{align*}
Additionally, both functions are bounded from below, \ie there exits a constant $c_0>0$, such that for almost every $t\in (0,T)$ and every $x \in \overline{\oe}$ it holds that
\begin{align*}
c_0 \le \jbteps (t,x) , \quad c_0 \le \jhteps (t,x).
\end{align*}
\end{lemma}
\begin{proof}
This is an easy consequence of $\eqref{EstimateGradJeps}$ and Lemma \ref{LemmaEstimatesTimeDisCoeff}.
\end{proof}
\begin{lemma}\label{AprioriRepsDt}
We have  $\rbteps \in L^{\infty}((0,T)\times \Omega)$ and  $\rhteps \in W^{1,\infty}((0,T),L^2(\Omega))$ such that $\rhteps , \partial_t \rhteps \in L^{\infty}((0,T)\times \Omega)$. For almost every $t \in (0,T)$ the functions $\rbteps, \, \rhteps,$ and $\partial_t \rhteps $ are constant on every micro cell $\epsilon (Y + k)$ for $k\in K_{\epsilon}$. Further, we have the estimate
\begin{align*}
\big\|\rbteps\big\|_{L^{\infty}((0,T)\times \Omega )} + \big\|\rhteps \big\|_{L^{\infty}((0,T)\times \Omega)} 
+ \big\|\partial_t \rhteps \big\|_{L^{\infty}((0,T)\times \Omega)} \le C\epsilon.
\end{align*}
\end{lemma}
\begin{proof}
This follows directly from Proposition \ref{PropEstimateDiscTimeReps}.
\end{proof}

\begin{lemma}\label{AprioriUepsDt}
For $\ubteps$ and $\uhteps$ we have the following \textit{a priori} estimates:
\begin{align*}
\big\|\ubteps\big\|_{L^{\infty}((0,T),L^2(\oe))} + \big\|\nabla \ubteps \big\|_{L^2((0,T),L^2(\oe))} &\le C,
\\
\big\| \uhteps \big\|_{L^{\infty}((0,T),L^2(\oe))} + \big\| \nabla \uhteps \big\|_{L^2((0,T),L^2(\oe))} + \epsilon \big\|\partial_t \uhteps \big\|_{L^2((0,T),H^1(\oe)')} &\le C.
\end{align*}
\end{lemma}
\begin{proof}
Again, we only consider the time-derivative, because the other inequalities are direct consequences of Proposition \ref{AprioriEstimateDiscreteSolutionUeps}. An elemental calculation gives us for $ t \in (t^{j-1},t^j)$
\begin{align*}
\partial_t \uhteps(t) = \frac{1}{\jjeps} \partial_t \whteps (t) - \frac{\ujoeps}{\jjeps} \frac{\jjeps - \jjoeps }{\dt}.
\end{align*}
Hence, for all $\phi \in H^1(\oe)$ with $\|\phi\|_{H^1(\oe)} \le 1$, we obtain
\begin{align*}
\left\langle \partial_t \uhteps (t),\phi \right\rangle_{H^1(\oe)',H^1(\oe)} &\le \big\|\partial_t \whteps(t)\big\|_{H^1(\oe)'} \left\|\frac{\phi}{\jjeps}\right\|_{H^1(\oe)} + \left\|\frac{\ujoeps}{\jjeps}\right\|_{L^2(\oe)} \left\|\frac{\jjeps - \jjoeps}{\dt}\right\|_{L^{\infty}(\oe)}
\\
&\le C \left(\epsilon^{-1} \big\|\partial_t \whteps(t)\big\|_{H^1(\oe)'} + 1 \right),
\end{align*}
where for the second inequality we used Lemma \ref{LemmaEstimatesTimeDisCoeff} and Proposition \ref{AprioriEstimateDiscreteSolutionUeps}, see also $\eqref{EstimateGradJeps}$. From Lemma \ref{AprioriEstimatesWeps} we obtain the desired result.
\end{proof}
Based on these uniform \textit{a priori} estimates  with respect to $\dt$, we pass to the limit $\dt \to 0$ to obtain a solution of the continuous microscopic model $\eqref{MicProbFixedDomain}$. 
\begin{corollary}\label{ConvergenceUepsdt}
There exists $\ueps \in L^2((0,T),H^1(\oe))\cap H^1((0,T),H^1(\oe)')$, such that up to a subsequence of $\{dt\}$ it holds for $\frac12 < \beta <  1$ 
\begin{align*}
\uhteps &\rightharpoonup \ueps &\mbox{ weakly in }& L^2((0,T),H^1(\oe)),
\\
\uhteps &\rightarrow \ueps &\mbox{ in }& L^2((0,T),H^{\beta}(\oe)),
\\
\partial_t \uhteps &\rightharpoonup \partial_t \ueps &\mbox{ weakly in }& L^2((0,T),H^1(\oe)'),
\\
\ubteps &\rightarrow \ueps &\mbox{ in }& L^2((0,T),L^2(\oe)),
\\
\nabla \ubteps &\rightharpoonup \nabla \ueps &\mbox{ weakly in }& L^2((0,T),L^2(\oe)),
\\
\ubteps &\rightarrow \ueps &\mbox{ in }& L^2((0,T),L^2(\geps)),
\end{align*}
for $\dt \to 0$. Further, we have the following \textit{a priori} estimate:
\begin{align*}
\Vert \ueps \Vert_{L^2((0,T),H^1(\oe))} + \epsilon \Vert \partial_t \ueps \Vert_{L^2((0,T),H^1(\oe)')} \le C.
\end{align*}
Additionally, we have $\ueps \in L^{\infty}((0,T),L^2(\oe))$ with
\begin{align*}
\Vert \ueps \Vert_{L^{\infty}((0,T),L^2(\oe))} \le C.
\end{align*}
\end{corollary}

\begin{proof}
From the \textit{a priori} estimates in Lemma \ref{AprioriUepsDt} we obtain the existence of functions $\ueps, \ueps^{\ast} \in L^2((0,T),H^1(\oe))$ and $\partial_t \ueps \in L^2((0,T),H^1(\oe)')$, such that up to a subsequence
\begin{align*}
\uhteps &\rightharpoonup \ueps  &\mbox{ weakly in }& L^2((0,T),H^1(\oe)),
\\
\uhteps &\rightharpoonup \ueps  &\mbox{ weakly$^{\ast}$ in }& L^{\infty}((0,T),L^2(\oe)),
\\
\ubteps &\rightharpoonup \ueps^{\ast} &\mbox{ weakly in }& L^2((0,T),H^1(\oe)),
\\
\partial_t \uhteps &\rightharpoonup \partial_t \ueps &\mbox{ weakly in }& L^2((0,T),H^1(\oe)').
\end{align*}
The \textit{a priori} estimates follow directly from the lower semi-continuity of the norm with respect to the weak- and weak$^{\ast}$-topology. Since for $\frac12 < \beta < 1$ the embedding $H^1(\oe) \hookrightarrow H^{\beta}(\oe)$ is compact, the Aubin-Lions Lemma implies the strong convergence of $\uhteps $ in $L^2((0,T),H^{\beta}(\oe))$. Let us check $\ueps = \ueps^{\ast}$. For $t \in (t^{j-1},t^j)$ it holds that
\begin{align*}
\big\|\uhteps(t) - \ubteps(t) \big\|^2_{L^2(\oe)} &\le 4\|\ujeps - \ujoeps \|^2_{L^2(\oe)}
\\
&= 4 \dt (\partial_t \uhteps , \ujeps - \ujoeps )_{L^2(\oe)}
\\
&\le 4 \dt \big\|\partial_t \uhteps (t) \big\|_{H^1(\oe)'} \|\ujeps - \ujoeps \|_{H^1(\oe)}.
\end{align*}
Integration with respect to time implies 
\begin{align*}
\big\|\uhteps - \ubteps \big\|^2_{L^2((0,T),L^2(\oe))} &\le 4 \dt \sum_{j=1}^N \int_{t^{j-1}}^{t^j} \big\|\partial_t \uhteps (t) \big\|_{H^1(\oe)'} \|\ujeps - \ujoeps \|_{H^1(\oe)} dt 
\\
&\le 2\dt \left( \big\|\partial_t \uhteps \big\|^2_{L^2((0,T),H^1(\oe)')} + \dt \sum_{j=1}^N \|\ujeps - \ujoeps \|^2_{H^1(\oe)} \right)
\\
&\le \frac{C \dt }{\epsilon} \overset{\dt \to 0}{\longrightarrow } 0,
\end{align*}
where in the last estimate we used the results from Proposition \ref{AprioriEstimateDiscreteSolutionUeps} and Lemma \ref{AprioriUepsDt}. The triangle inequality implies $\ueps = \ueps^{\ast}$ and  the strong convergence of $\ubteps$ in $L^2((0,T),L^2(\oe))$.

It remains to show the strong convergence of the traces. This follows directly from the following interpolated trace inequality: There exists a constant $C_{\epsilon} >0$ (which may depend on $\epsilon$), such that for all $v \in H^1(\oe)$ it holds that
\begin{align*}
\|v \|_{L^2(\geps)}^2 \le C_{\epsilon} \left( \|v\|^2_{L^2(\oe)} + \|v\|_{L^2(\oe)}\|v\|_{H^1(\oe)}\right).
\end{align*}
Now, the strong convergence of $\ubteps$ in $L^2((0,T),L^2(\oe))$ implies the strong convergence of the traces in $L^2((0,T),L^2(\geps))$.
\end{proof}

\begin{corollary}\label{ConvergenceRepsDt}
There exists $\reps \in W^{1,\infty}((0,T),L^2(\Omega))$ constant on every micro cell $\epsilon(Y + k)$ for $k\in K_{\epsilon}$ and $\reps, \partial_t \reps \in L^{\infty}(0,T)\times \Omega)$, such that up to a subsequence of $\{dt\}$ it holds for every $p\in [1,\infty)$ 
\begin{align*}
\rhteps &\overset{\ast}{\rightharpoonup} \reps &\mbox{ weakly}^{\ast} \mbox{ in }& W^{1,\infty}((0,T),L^2(\Omega)),
\\
\rbteps &\rightarrow \reps &\mbox{ in }& L^{\infty}((0,T)\times \Omega),
\\
\rbteps &\rightarrow \reps &\mbox{ in }& L^{\infty}((0,T)\times \geps),
\end{align*}
for $\dt \to 0$ and all $\alpha \in (0,1)$. Further, we have
\begin{align*}
\Vert \reps \Vert_{W^{1,\infty}((0,T),L^{2}(\Omega))} + \Vert \reps \Vert_{L^{\infty}((0,T)\times \Omega)} + \Vert \reps \Vert_{L^{\infty}((0,T)\times \geps)} \le C \epsilon.
\end{align*}
\end{corollary}
\begin{proof}
The first convergence follows directly from the \textit{a priori} estimates in Lemma \ref{AprioriRepsDt}.
Further, due to the Arzel\`{a}-Ascoli theorem, $\rhteps$ converges to $\reps$ in $C^{0,\alpha}([0,T],L^{\infty}(\Omega))$ for arbitrary $\alpha \in (0,1)$. We emphasize that for every $t \in [0,T]$ the set $\{\rhteps(t)\}_{\dt}$ is a finite dimensional subset of $L^{\infty}(\Omega)$ and is therefore relatively compact. The convergence in $C^{0,\alpha}([0,T],L^{\infty}(\Omega))$ implies the strong convergence of $\rhteps $ in $L^{\infty}((0,T)\times \Omega)$ and, since $\rhteps $ is constant on every micro cell, in $L^{\infty}((0,T)\times \geps)$.
  Let us check the strong convergence of $\rbteps$.  We have
\begin{align*}
\big\|\rbteps(t, \cdot ) - \rhteps (t,\cdot)\big\|_{L^{\infty}((0,T)\times \Omega)} \le 2 \max_{j\in \{1,\ldots,N\}} \big\| \rjeps - \rjoeps \big\|_{L^{\infty}(\Omega)} \le C\dt.
\end{align*}
The triangle inequality implies $\rbteps \rightarrow \reps$ in $L^{\infty}((0,T)\times \Omega)$.  The inequality follows from Lemma \ref{AprioriRepsDt}.
\end{proof}
We will prove later in Lemma \ref{LemmaStrongConvergenceTimeDerRepsRothe} that also the time derivative of $\partial_t \rhteps $ converges strongly in the $L^p$-sense. Now, let us define the function
\begin{align}\label{FormulaRepresentationReps}
\jeps(t,x) = \det \left( I_n + \frac{\reps(t,x) - \epsilon \oR}{\epsilon} \nabla_y (\chi_0 \nu_0 )\left(\fxe\right)\right).
\end{align}
In the following corollary we prove the convergence of $\jbteps$ and $\jhteps$ to the function $\jeps$ for $\dt \to 0$.

\begin{corollary}\label{CorollaryConvergenceJdteps}
Up to a subsequence, for every $p \in (1,\infty)$ there holds the following convergence
\begin{align*}
\jbteps &\rightarrow \jeps \quad &\mbox{ in }& L^{\infty}((0,T)\times \oe),
\\
\jhteps &\rightharpoonup \jeps &\mbox{ weakly in }& L^p((0,T),W^{1,p}(\oe)),
\\
\jhteps &\rightarrow \jeps &\mbox{ in }& L^{\infty}((0,T)\times \oe),
\\
\partial_t \jhteps &\rightharpoonup \partial_t \jeps &\mbox{ weakly in }& L^p((0,T)\times \oe)
\end{align*}
for $\dt \to 0$.  We emphasize that, due to the structure of $\jeps$ in $\eqref{FormulaRepresentationReps}$, we have $\jeps \in C^0([0,T]\times \overline{\oe})$. Further, the following inequality holds
\begin{align*}
\Vert \jeps \Vert_{C^0([0,T]\times \overline{\oe})} + \Vert \partial_t \jeps \Vert_{L^{\infty}((0,T)\times \oe)} + \epsilon \Vert \nabla \jeps \Vert_{C^0([0,T]\times \overline{\oe})} \le C.
%
\end{align*}
\end{corollary}
\begin{proof}
First of all, let us check the convergence of $\jbteps$. Using the Lipschitz continuity of the determinant and the essential boundedness of $\rbteps$ with respect to $\dt$, we obtain
\begin{align*}
\big\|\jeps - \jbteps \big\|_{L^{\infty}((0,T)\times \oe)} \le \frac{C}{\epsilon} \big\|\reps - \rbteps \big\|_{L^{\infty}((0,T)\times \oe)}  \overset{\dt \to 0 }{\longrightarrow } 0.
\end{align*}
Further, we have  from $\eqref{EstimatesTimeDiscCoeffDtJ}$
\begin{align*}
\Vert \jhteps - \jbteps \Vert_{L^{\infty}((0,T)\times \oe)} \le 2 \Vert \jjeps - \jeps^{j-1} \Vert_{L^{\infty}((0,T)\times \oe)} \le C \dt.
\end{align*}
The weak convergence of $\jhteps $ in $L^p((0,T),W^{1,p}(\oe)) \cap W^{1,p}((0,T),L^p(\oe))$ now follows directly from the \textit{a priori} estimates in Lemma \ref{AprioriEstimatesJepsDt}. The estimates for $\jeps$ are a consequence of Lemma \ref{AprioriEstimatesJepsDt}, formula $\eqref{FormulaRepresentationReps}$, and the inequality for $\reps$ in Corollary \ref{ConvergenceRepsDt}.

\end{proof}

\begin{corollary}
Up to a subsequence, we have the following convergence results:
\begin{align*}
\whteps &\rightharpoonup \jeps \ueps &\mbox{ in }& L^2((0,T),H^1(\oe)),
\\
\partial_t \whteps &\rightharpoonup \partial_t \big(\jeps \ueps\big) &\mbox{ in }& L^2((0,T),H^1(\oe)'),
\\
\wbteps &\rightarrow \jeps \ueps &\mbox{ in }& L^2((0,T),L^2(\oe))
\end{align*}
for $\dt \to 0$. Further, we have the \textit{a priori} estimate
\begin{align*}
\Vert \jeps \ueps \Vert_{L^{\infty}((0,T),L^2(\oe))} + \epsilon  \Vert \nabla  (\jeps \ueps ) \Vert_{L^2((0,T),L^2(\oe))} + \Vert \partial_t (\jeps \ueps) \Vert_{L^2((0,T),H^1(\oe)')} \le C.
\end{align*}
\end{corollary}
\begin{proof}
The existence of a limit function $W_{\epsilon}$ such that the convergence results above hold (up to a subsequence) for $\jeps \ueps$ replaced by $W_{\epsilon}$ follows as in the proofs above. The identity $W_{\epsilon} = \jeps \ueps$ follows easily from the pointwise almost everywhere convergence of $\jhteps \uhteps$ to $\jeps \ueps $ (for a subsequence) and $\whteps = \jhteps \uhteps$.
\end{proof}

\subsection{Existence for the continuous model}

Now, we want to pass to the limit $\dt \to 0$ in the time-discrete problem from Definition \ref{DefTimeDiscModel}. Let us define the piecewise constant coefficients in the following way for $t \in (t^{j-1},t^j)$:
\begin{align*}
&\gbteps(t,x) := G_{\epsilon}^{j-1}(x), \quad
\sbteps(t,x) := \seps^j(x),\quad 
\fbteps(t,x) := \nabla \seps^j(x),
\\
&\dbteps(t,x) := D_{\epsilon}^j, \quad
\vbteps(t,x) := V_{\epsilon}^j, \quad
f^{\epsilon}_{\dt} := f_{\epsilon}^j.
\end{align*}
Multiplying the  variational equation $\eqref{VarEquTimeDiscModel} $ with $\phi \in C^{\infty}_0\big([0,T)\times \overline{\oe}\big)$ and integrating with respect to time gives us:
\begin{align}
\begin{aligned}
\label{DiskVarEquationInterpolated}
\int_0^T \int_{\oe} \partial_t& \whteps \phi dx dt + \int_0^T \int_{\oe} \dbteps \nabla \ubteps \cdot \nabla \phi dxdt
\\
=&\int_0^T \int_{\oe} \vbteps \cdot \nabla \ubteps \phi dx dt + \int_0^T \int_{\oe} \jbteps f_{\dt}^{\epsilon} \phi dx dt
\\
&+ \int_0^T \int_{\oe} \ubteps \partial_t \jhteps \phi dx dt - \int_{\geps} \gbteps \big(\ubteps(\cdot - \dt) - \rho \big) \jbteps \phi d\sigma dt.
\end{aligned}
\end{align}
To pass to the limit $\dt \to 0$ in $\eqref{DiskVarEquationInterpolated}$, we need the following convergence results for the interpolated coefficients:
\begin{lemma}\label{LemmaStrongConvergenceTimeDerRepsRothe}
Up to a subsequence of $\{dt\}$ we have the following convergence results for arbitrary $p \in [1,\infty)$
\begin{subequations}
\begin{align}
\label{ConvergenceGepsDt}
\partial_t \rhteps = \gbteps &\rightarrow G_{\epsilon}(\ueps,\reps) &\mbox{ in }& L^p((0,T)\times \geps),
\\
\label{ConvergenceDepsDt}
\dbteps &\rightarrow D_{\epsilon} &\mbox{ in }& L^p((0,T)\times \oe),
\\
\label{ConvergenceVepsDt}
\vbteps &\rightarrow V_{\epsilon} &\mbox{ in }& L^p((0,T)\times \oe)
\end{align}
\end{subequations}
for $\dt \to 0 $.
\end{lemma}
\begin{proof}
We start with the convergence $\eqref{ConvergenceGepsDt}$. 
 We have  
\begin{align*}
\big\| \gbteps - G_{\epsilon}(\ueps, \reps)&\big\|^2_{L^2((0,T)\times \geps)} 
\\
&\le C\epsilon^2 \sum_{j=1}^N \sum_{{\bf k} \in K_{\epsilon}} \int_{t^{j-1}}^{t^j}\int_{\geps\left(\epsilon k\right)} \left| g \left(\ujoeps , \frac{\rjoeps}{\epsilon}\right) - g \left(\ueps , \frac{\reps}{\epsilon}\right) \right|^2d\sigma dt.
\\
&= C\epsilon^2 \sum_{j=1}^n \int_{t^{j-1}}^{t^j} \left\|g\left( \ujoeps , \frac{\rjoeps}{\epsilon}\right) - g\left(\ueps , \frac{\reps}{\epsilon}\right)\right\|^2_{L^2(\geps)} dt.
\end{align*}
Using the Lipschitz continuity of $g$, we obtain
\begin{align*}
\big\| \gbteps& - G_{\epsilon}(\ueps, \reps)\big\|^2_{L^2((0,T)\times \geps)} 
\\
&\le C \epsilon^2  \left(\big\|\ubteps (\cdot - \dt ) - \ubteps \big\|^2_{L^2((0,T)\times \geps)} + \big\|\rbteps (\cdot - \dt ) - \rbteps \big\|^2_{L^2((0,T)\times \geps)} \right).
\end{align*}
We emphasize that for $t \kl 0$ we have $\ubteps = \uepsin$ and $\rbteps = \repsin$.
Because of the convergence results from Corollary \ref{ConvergenceUepsdt} and \ref{ConvergenceRepsDt}, as well as the Kolmogorov compactness theorem, the right-hand side tends to $0$ for $\dt \to 0$. Hence, we have $\partial_t \rhteps = \gbteps \rightarrow G_{\epsilon}(\ueps , \reps)$ in $L^2((0,T)\times \geps)$ for $\dt \to 0$. Since $\gbteps \in L^{\infty}((0,T)\times \geps)$, the dominated convergence theorem of Lebesgue implies that the convergence also holds up to a subsequence in $L^p((0,T)\times \geps)$ for all $p \in (1,\infty)$. 

For the convergence $\eqref{ConvergenceDepsDt}$ we first notice that
\begin{align*}
\dbteps = \jbteps \left[\fbteps \right]^{-1} D(\sbteps) \left[\fbteps\right]^{-1},
\end{align*}
and therefore, due to the convergence results from Corollary \ref{ConvergenceRepsDt} and \ref{CorollaryConvergenceJdteps}, we obtain $\dbteps \rightarrow D_{\epsilon}$ almost everywhere in $(0,T)\times \oe$ up to a subsequence. Since $\dbteps$ is essential bounded, the dominated convergence theorem of Lebesgue implies $\dbteps \rightarrow D_{\epsilon}$ in $L^p((0,T)\times \oe)$. The last convergence $\eqref{ConvergenceVepsDt}$ follows by similar arguments, where we here  additionally  use the strong convergence (and therefore the pointwise almost everywhere convergence up to a subsequence) of $\partial_t \rhteps$ proved above. We emphasize that due to the strong convergence of $\partial_t \rhteps$ is also valid in $L^p((0,T)\times \Omega)$, since it is constant on every micro cell $\epsilon (Y + {\bf k})$.
\end{proof}

As a consequence of the previous results, we can easily prove the existence result in Theorem \ref{ExistenceTheorem}:
\begin{proof}[Proof of Theorem \ref{ExistenceTheorem}]
This follows directly by passing to the limit in the variational equation $\eqref{DiskVarEquationInterpolated}$ and in $\eqref{DefMicProblFixedDomainODEReps}$. To establish the initial condition, we just integrate by parts with respect to time in $\eqref{MicProbFixedDomain}$ and $\eqref{DiskVarEquationInterpolated}$. Since this is standard, we skip the details.
\end{proof}
In the following Lemma we summarize the \textit{a priori} estimates for the weak solution $(\ueps,\reps)$ which will be necessary for the derivation of the macroscopic model:
\begin{lemma}\label{AprioriEstimatesSummary}
A constant $C > 0$ not depending on $\epsilon$ exists such that the weak solution $(\ueps,\reps)$ from Theorem \ref{ExistenceTheorem} satisfies the \textit{a priori} estimates
\begin{align*}
\epsilon \Vert \partial_t \ueps \Vert_{L^2((0,T),H^1(\oe)')} + \Vert \ueps \Vert_{L^{\infty}((0,T),L^2(\oe))}  +  \Vert \ueps \Vert_{L^2((0,T),H^1(\oe))} &\le C,
\\
\frac{1}{\epsilon}\Vert \reps \Vert_{L^{\infty}((0,T)\times \oe)} + \frac{1}{\epsilon}\Vert \partial_t \reps \Vert_{L^{\infty}((0,T)\times \oe)} &\le C,
\\
\Vert \partial_t (\jeps \ueps) \Vert_{L^2((0,T),H^1(\oe)')} + \Vert \jeps \ueps \Vert_{L^{\infty}((0,T),L^2(\oe))} \qquad & \\
+ \epsilon \Vert \nabla (\jeps \ueps)\Vert_{L^2((0,T),L^2(\oe))} &\le C.
\end{align*}
Also, for the Hanzawa transformation $\seps$ and the Jacobi-determinant $\jeps$ one has 
\begin{align*}
\Vert \seps \Vert_{L^{\infty}((0,T)\times \oe)} + \Vert \nabla \seps \Vert_{L^{\infty}((0,T)\times \oe)} + \frac{1}{\epsilon}\Vert \partial_t \seps \Vert_{L^{\infty}((0,T)\times \oe)} &\le C,
\\
\Vert \jeps \Vert_{L^{\infty}((0,T)\times \oe)} + \Vert \partial_t \jeps \Vert_{L^{\infty}((0,T)\times \oe)} + \epsilon \Vert \nabla \jeps \Vert_{L^{\infty}((0,T)\times \oe)} &\le C.
\end{align*}
Finally, there exists a constant $c_0 > 0$ such that (independently of $\epsilon$) 
\begin{align*}
c_0 \le \jeps(t,x) \quad \mbox{ for almost every } (t,x) \in (0,T)\times \oe.
\end{align*}
\end{lemma}

Additionally, the differential equation $\eqref{MicroscopicModel_ODE}$ for $\reps$ and the boundedness of $g$ immediately implies with the mean value theorem the following uniform estimates:
\begin{corollary}
\label{HoelderEstimatesRepsJeps}
It holds that
\begin{align*}
    \Vert \reps \Vert_{C^{0,1}([0,T],L^{\infty}(\Omega))}  +  \Vert \jeps \Vert_{C^{0,1}([0,T],L^{\infty}(\Omega))}  \le C.
\end{align*}
\end{corollary}

\section{Derivation of the macroscopic model}
\label{SectionDerivationMacroModel}

In this section we show that the sequence of micro solutions $(\ueps,\reps)$ converges in a suitable sense to the solution of the macro-model $\eqref{MacroModelStrong}$. For this we prove two-scale compactness results which allow to pass to the limit in the variational equation for the micro model $\eqref{WeakFormulationMicroscopicModel}$. However, the \textit{a priori} estimates in Lemma \ref{AprioriEstimatesSummary} only guarantee weak convergence for $\ueps$ and $\reps$, and this is not enough to pass to limit in $\eqref{WeakFormulationMicroscopicModel}$, since also the strong convergence of the  coefficients, which depend in a nonlinear way on the micro solution itself, is needed. Also to pass to the limit in the boundary term and in equation $\eqref{DefMicProblFixedDomainODEReps}$ strong convergence is necessary. We refer the reader to the Appendix \ref{SectionTwoScaleConvergence} for a short overview on the two-scale convergence.


\subsection{Compactness results}

 The crucial point is to obtain the strong convergence of $\ueps$, $\reps$, and  $\partial_t \reps$. Standard arguments for $\ueps$ of Aubin-Lions-type, see \cite{MeirmanovZimin}, fail, since we have no uniform bound for the time derivative $\partial_t \ueps$ with respect to $\epsilon$ (remember, the norm is of order $\epsilon^{-1}$).  Hence, it makes sense to consider also the sequence $\jeps \ueps$ which has bounded time derivative, but  which gradient behaves badly with respect to $\epsilon$. We will see that together with the uniform bounds on $\jeps$ this is enough to guarantee to strong convergence of $\ueps$ via a variational argument. The strong convergence of $\reps$ is obtained by using the Kolmogorov-compactness theorem, see \cite{Brezis}, based on uniform bounds for the difference of small shifts. Combining the strong two-scale convergence of $\ueps$ and $\reps$, we are able to show the strong convergence of $\partial \reps$ and identify its limit.
Before proceeding, we observe that, since $\oe$ is connected, there exists an extension of $\ueps$ to the whole domain $\Omega$ (also denoted $\ueps$ to avoid an excess of notations) such that (\cite{Acerbi1992,CioranescuSJPaulin}, see Section \ref{SubsectionStrongConvergenceUeps} for more details)
\begin{align}\label{AprioriExtension}
\Vert \ueps \Vert_{L^2((0,T),H^1(\Omega))} \le C.
\end{align}

\subsubsection{Weak two-scale compactness results}

We start with some (weak) two-scale convergence results which follow directly from the \textit{a priori} estimates in Lemma \ref{AprioriEstimatesSummary}. For the definition and some basic compactness results for the two-scale convergence see the Appendix \ref{SectionTwoScaleConvergence}:

\begin{corollary}\label{WeakTSCompactnessResults}\
Up to a subsequence the following convergence results for the microscopic solutions $(\ueps,\reps)$ of problem $\eqref{MicroscopicModel}$ are valid:
\begin{enumerate}
[label = (\roman*)]
\item There exist $u_0 \in L^2((0,T),H^1(\Omega))$ and $u_1 \in L^2((0,T)\times \Omega,H^1_{\per}(Y^{\ast})/\R))$, such that  
\begin{align*}
\ueps &\rightarrow u_0 &\mbox{ in the two-scale sense in }  L^2,
\\
\nabla \ueps &\rightarrow \nabla_x u_0 + \nabla_y u_1 &\mbox{ in the two-scale sense in } L^2,
\\
\ueps &\rightharpoonup^* u_0 &\mbox{ weakly}^* \mbox{ in } L^{\infty}((0,T),L^2(\Omega)).
\end{align*}
\item There exists $R_0 \in W^{1,\infty}((0,T),L^2(\Omega))$ with $R_0 \in L^{\infty}((0,T)\times \Omega)$ and $\partial_t R_0 \in L^{\infty}((0,T)\times \Omega)$, such that for every $p \in [1,\infty)$
\begin{align*}
\epsilon^{-1} \reps &\rightharpoonup^* R_0  &\mbox{ weakly}^* \mbox{ in } W^{1,\infty}((0,T),L^2(\Omega)),
\\
\epsilon^{-1} \reps &\rightarrow R_0 &\mbox{ in the two-scale sense in } L^p,
\\
\epsilon^{-1} \partial_t\reps &\rightarrow \partial_t R_0 &\mbox{ in the two-scale sense in } L^p.
\end{align*}
Additionally the two-scale convergences above also hold in the two-scale sense on $\geps$ in $L^p$. 
\item Defining  
\begin{align*}
S_0(t,x,y) &:= y + \big(R_0(t,x) - \oR\big) (\chi_0 \nu_0) (y),
\end{align*}
it holds for every $p \in (1,\infty)$  that
\begin{align*}
\seps &\rightarrow x &\mbox{ in } L^{\infty}((0,T)\times \Omega),
\\
\nabla \seps &\rightarrow \nabla_y S_0 &\mbox{ in the two-scale sense in } L^p.
\end{align*}
\end{enumerate}
\end{corollary}
\begin{proof}
These results follow directly from the \textit{a priori} estimates in Lemma \ref{AprioriEstimatesSummary} and standard two-scale compactness results from Lemma \ref{BasicTwoScaleCompactness} in the Appendix \ref{SectionTwoScaleConvergence}. We point out that $R_0$ is independent of the microscopic variable $y \in Y$, since $\reps$ is constant on every microscopic cell $\epsilon (Y + {\bf k})$ with ${\bf k} \in K_{\epsilon}$.
\end{proof}

\subsubsection{Strong convergence of $\reps$}

Next, we prove the strong convergence of $\epsilon^{-1} \reps$ in the $L^p$-sense, where we make use of the Kolmogorov-compactness theorem, see for example \cite{Brezis}. We first introduce the following notation for $ 0 \kl h$: 
\begin{align*}
\Omega^h := \{x \in \Omega \, : \, \mathrm{dist}(x,\partial \Omega) \gr h\}.
\end{align*}
Further, for $\Bell \in \Z^n$ we define
\begin{align*}
K_{\epsilon}^{\Bell} := \left\{ {\bf k} \in \Z^n \, : \, \epsilon(Y + {\bf k} + {\Bell}) \cup \epsilon(Y + {\bf k}) \subset \Omega \right\},
\end{align*}
and 
\begin{align*}
\oe^{\Bell}:= \mathrm{int}\bigcup_{{\bf k}\in K_{\epsilon}^{\Bell}} \epsilon (\overline{Y^{\ast}} + {\bf k}),	\qquad\geps^{\Bell} := \bigcup_{{\bf k}\in K_{\epsilon}^{\Bell} } \epsilon (\Gamma + {\bf k}).
\end{align*}
We emphasize that $\oe^{\Bell} \subset \Omega^{\vert \epsilon {\Bell}\vert}$ and  for $\vert \epsilon {\Bell} \vert \kl \frac{h}{2}$ and $\epsilon $ small enough it holds that 
\begin{align}\label{SubsetOmegah}
\Omega^h \subset \mathrm{int} \bigcup_{{\bf k}\in K_{\epsilon}^{\Bell}} \epsilon(\overline{Y}  + {\bf k}).
\end{align}

\begin{proposition}\label{StrongConvergenceReps}
For all $p \in [1,\infty)$ it holds up to a subsequence  that
\begin{align*}
\epsilon^{-1}\reps \rightarrow R_0 \quad \mbox{ in } L^p((0,T)\times \Omega).
\end{align*}
Additionally, we have $R_0 \in L^{\infty}((0,T),H^1(\Omega))$.
\end{proposition}

\begin{proof}
It is enough to show that  $\reps \rightarrow R_0$ in $L^2((0,T)\times \Omega)$. Then, the desired result follows from the dominated convergence theorem of Lebesgue. We use the Kolmogorov-compactness theorem. Let $s \in \R$ and $\xi \in \R^n$, then we have to show:
\begin{align*}
\sup_{\epsilon > 0} \epsilon^{-1}\Vert \reps(t + s,x + \xi) - \reps(t,x) \Vert_{L^2((0,T) \times \Omega)} \overset{(s,\xi)\to 0}{\longrightarrow} 0, 
\end{align*}
where we can extend $\reps$ by zero to $\R \times \R^n$. Since $\epsilon^{-1}\reps$ is uniformly bounded with respect to $\epsilon$ in $L^{\infty}((0,T)\times \Omega)$, it is enough to show for $0 \kl h \ll 1$ and $\vert s \vert +\vert \xi \vert \kl \frac{h}{2}$ that
\begin{align*}
\sup_{\epsilon > 0} \epsilon^{-1} \Vert \reps(t + s , x + \xi) - \reps(t,x)\Vert_{L^2((h,T-h)\times \Omega^h)} \overset{(s,\xi)\to 0}{\longrightarrow } 0.
\end{align*}
 First of all, we have
\begin{align*}
\Vert &\reps(t + s, x + \xi) - \reps(t,x)\Vert_{L^2((h,T-h)\times \Omega^h)} 
\\
&\le \Vert \reps(t+s,x) - \reps(t,x)\Vert_{L^2((h,T-h)\times \Omega) } + \Vert \reps(t,x+\xi) - \reps(t,x) \Vert_{L^2((0,T)\times \Omega^h)}.
\end{align*}
For the first term we obtain from the essential boundedness of $g$:
\begin{align*}
\epsilon^{-2} &\Vert \reps(t +s , x) - \reps(t,x)\Vert^2_{L^2((h,T-h) \times \Omega)}
\\
&= \int_0^{T-s} \frac{1}{\vert\epsilon \Gamma\vert^2} \left( \int_t^{t+s} \int_{\geps\left(t,\epsilon \bfxe\right)} g\left(\ueps(\tau,z) , \frac{\reps(\tau,z)}{\epsilon}\right) dz d\tau \right)^2 dt 
\\
&\le C s^2 \ \Vert g\Vert_{L^{\infty}(\R \times \R^n)}^2.
\end{align*}
Let us consider the term including the shifts in the spatial variable. 
We have (see \cite[Proof of Theorem 13]{Gahn}) for $t \in (0,T)$
\begin{align}\label{AuxiliaryInequalityShiftsDisplacement}
\Vert \reps (t,x + \xi ) - \reps(t,x) \Vert_{L^2( \Omega^h)} \le \sum_{{\bf j}\in\{0,1\}^n} \left\Vert \reps\left(t,x+ \epsilon {\bf j} + \epsilon \left[\frac{\xi}{\epsilon}\right]\right) -\reps (t,x) \right\Vert_{L^2(\Omega^h)}.
\end{align}
Let $\Bell:= \Bell(\epsilon, {\bf j}):= {\bf j} +  \left[\dfrac{\xi}{\epsilon}\right]$. Then for  $t \in (0,T)$ we obtain with $\eqref{DefMicProblFixedDomainODEReps}$
\begin{align*}
\epsilon^{-2} \Vert \reps(t,x + \epsilon \Bell)& - \reps (t,x) \Vert^2_{L^2(\Omega^h)} \le C \epsilon^{-2} \Vert \repsin(x + \Bell \epsilon)- \repsin \Vert^2_{L^2(\Omega^h)}
\\
 +  C\int_{\Omega^h } \bigg[ \frac{1}{\vert\epsilon \Gamma\vert} & \int_0^t  \int_{\geps\left(t,\epsilon \bfxe\right)} g\left(\ueps(s,z + \epsilon l ) , \frac{\reps(s,z + \epsilon l)}{\epsilon}\right)   \\
&  -   g\left(\ueps(s,z),\frac{\reps(s,z)}{\epsilon}\right) dz ds \bigg]^2 dx
=: A_{\epsilon}^1 + A_{\epsilon}^2.
\end{align*}
For the first term $A_{\epsilon}^1$ including the initial value we immediately obtain from the Assumption \ref{AssumptionInitialValueReps} that $A_{\epsilon}^1 \le C \vert {\Bell} \epsilon\vert^2$. For $A_{\epsilon}^2$, using the trace inequality $\eqref{ScaledTraceInequality}$, the Lipschitz continuity of $g$, $\eqref{SubsetOmegah}$, and that $\reps$ constant on every micro cell gives
\begin{align*}
A_{\epsilon}^2 &\le \frac{C}{\vert\epsilon \Gamma \vert}   \int_{\Omega^h} \int_0^t \int_{\epsilon \Gamma} \left\vert \ueps\left(s,z + \epsilon \left[\fxe\right] + \epsilon {\Bell} \right)  - \ueps\left(s,z + \epsilon\left[\fxe\right] \right) \right\vert^2
\\
&\hspace{5em} + \epsilon^{-2}  \left\vert \reps\left(s,z + \epsilon \left[\fxe\right] + \epsilon {\Bell} \right) - \left(s,z + \epsilon\left[\fxe\right] \right)\right\vert^2 d\sigma_z ds dx 
\\
&\le  \sum_{{\bf k} \in K_{\epsilon}^{\Bell}} \bigg\{\frac{C}{\vert \epsilon \Gamma\vert}   \epsilon^{n} \int_0^t \int_{\epsilon \Gamma} \left\vert \ueps(s,z + \epsilon {\bf k} + \epsilon {\Bell}) - \ueps(s,z+\epsilon {\bf k} )\right\vert^2 d\sigma_z ds
\\
&\hspace{2em} + \frac{C}{\vert \epsilon \Gamma\vert \epsilon^2} \int_{\epsilon(Y+{\bf k}) \cap \Omega^h} \int_0^t \int_{\epsilon\Gamma} \left\vert \reps(s,z+\epsilon {\bf k} + \epsilon {\Bell}) - \reps(s,z+\epsilon {\bf k})\right\vert^2 d\sigma_z ds dx \bigg\}
\\
&= C \epsilon \int_0^t \int_{\geps^{\Bell}}  \vert \ueps(s, z + \epsilon {\Bell} ) - \ueps(s,z) \vert^2 d\sigma_z ds
\\
& \hspace{5em} + C\int_0^t \int_{\Omega^h }   \epsilon^{-2} \vert \reps(s,z+\epsilon {\Bell}) - \reps(s,z) \vert^2 dz ds . 
\end{align*}
Therefore, 
\begin{align*}
A_{\epsilon}^2 &\le C \int_0^t \int_{\oe^{\Bell}}  \vert \ueps(s, z + \epsilon {\Bell} ) - \ueps(s,z) \vert^2  + \epsilon^2\vert \nabla \ueps (s, z + \epsilon {\Bell} ) - \nabla \ueps (s,z)\vert^2 dz ds 
\\
&\hspace{5em}+ \int_0^t \int_{\Omega^h} \epsilon^{-2} \vert \reps(s,z+\epsilon {\Bell}) - \reps(s,z) \vert^2 dz ds 
\\
&\le C \int_0^t \int_{\Omega^{\vert \epsilon {\Bell}\vert}} \int_0^1 \vert \nabla \ueps(s,z + \lambda \epsilon {\Bell} )\vert^2 \cdot \vert \epsilon {\Bell} \vert^2 d\lambda dz ds \\
& \qquad + C\epsilon^2 + C \epsilon^{-2} \Vert \reps(\cdot,\cdot + \epsilon {\Bell} ) - \reps\Vert^2_{L^2((0,t)\times \Omega^h)}. 
\end{align*}
We can now use the \textit{a priori} estimate for $\nabla \ueps$, as proved in  Lemma \ref{AprioriEstimatesSummary}. Although these are obtained for the perforated domain $\oe$, they remain valid for the extension of $\ueps$ to $\Omega$, see $\eqref{AprioriExtension}$. From the above, we obtain 
\begin{align*}
A_{\epsilon}^2 &\le C \left( \Vert \nabla \ueps \Vert_{L^2((0,t)\times \Omega)}^2 \vert \epsilon {\Bell}\vert^2  + C\epsilon^2 + \epsilon^{-2} \Vert \reps(\cdot,\cdot + \epsilon {\Bell} ) - \reps\Vert^2_{L^2((0,t)\times \Omega^h)}\right)
\\
&\le C \left( \vert \epsilon \Bell \vert^2 + C\epsilon^2 + \epsilon^{-2} \Vert \reps(\cdot,\cdot + \epsilon \Bell ) - \reps\Vert^2_{L^2((0,t)\times \Omega^h)}\right). 
\end{align*}
The Gronwall-inequality gives 
\begin{align*}
\epsilon^{-2} \Vert \reps(t , x + \epsilon {\Bell} ) - \reps(t,x) \Vert^2_{L^{\infty}((0,T), L^2(\Omega^h))} \le C  \big(\vert \epsilon {\Bell} \vert^2 + \epsilon^2\big).
\end{align*}
With ${\Bell} = {\bf j} +  \left[\dfrac{\xi}{\epsilon}\right]$ and $\eqref{AuxiliaryInequalityShiftsDisplacement}$ we obtain
\begin{align}\label{InequalityShiftsSpaceReps}
\epsilon^{-1} \Vert \reps(t,x + \xi) - \reps (t,x) \Vert_{L^{\infty}((0,T),L^2(\Omega^h))} \le C \left(\vert \xi \vert + \epsilon \right).
\end{align}
Hence, for arbitrary $\rho \gr 0$ there exists $\epsilon_0, \delta_0 \gr 0$, such that, for all $\xi \in \R^n$ with $\vert \xi \vert \kl \delta_0$ it holds that
\begin{align}\label{AuxiliaryInequalityKolmogorov}
\sup_{\epsilon \le \epsilon_0} \epsilon^{-1} \Vert \reps(t,x+\xi) - \reps(t,x) \Vert_{L^{\infty}((0,T),L^2(\Omega^h))} \le \rho.
\end{align}
Since the sequence $\epsilon$ is countable there are only finitely many $\epsilon$ (denoted by $\epsilon_1, \ldots, \epsilon_N$) such that $\epsilon_i \gr \epsilon_0$ for $i=1,\ldots, N$. Due to the Kolmogorov-compactness theorem there exists $\delta_i\gr 0$ such that for all $\vert \xi \vert \kl \delta_0$
\begin{align*}
\epsilon^{-1} \Vert R_{\epsilon_i}(t,x+\xi) - R_{\epsilon_i}(t,x) \Vert_{L^{\infty}((0,T),L^2(\Omega^h))} \le \rho.
\end{align*}
Hence, $\eqref{AuxiliaryInequalityKolmogorov}$ is also valid for all $\vert \xi \vert \kl \delta:=\min\{\delta_0,\ldots, \delta_N\}$ and if we take the supremum over all $\epsilon$.
This gives the strong convergence of $\epsilon^{-1}\reps $ to $R_0$ in $L^2((0,T)\times \Omega)$. 

It remains to establish the higher regularity of $R_0$ with respect to $x$. This is an easy consequence of $\eqref{InequalityShiftsSpaceReps}$.
In fact, using the strong convergence (up to a subsequence) of $\reps$ in $L^p((0,T)\times \Omega)$ for $p \in [1,\infty)$ showed above, we obtain with $\vert \xi \vert \kl \frac{h}{2}$  for $\epsilon \to 0$ in $\eqref{InequalityShiftsSpaceReps}$
\begin{align*}
\left\Vert \frac{R_0(t,x + \xi) - R_0 }{\vert \xi \vert } \right\Vert_{L^p((0,T),L^2(\Omega^h))} \le C,
\end{align*}
with a constant $C \gr 0$ which can be chosen independently of $p$.  This implies $R_0 \in L^p((0,T),H^1(\Omega))$ with $L^p$-norm bounded uniformly with respect to $p$ and the proposition is proved.
\end{proof}

\begin{remark}\label{RemarkKonvergenzReps}\mbox{}
\begin{enumerate}
[label = (\roman*)]
\item The above proof shows that we also have the pointwise convergence 
\begin{align*}
\epsilon^{-1} \reps(t,\cdot) \rightarrow R_0(t,\cdot) \quad \mbox{ in } L^p(\Omega)
\end{align*}
for almost every $t \in (0,T)$ and $1 \le p \kl \infty$.
\item In the same way as we proved the regularity of $R_0$, we obtain the regularity of the initial condition $R^0 \in H^1(\Omega)$ from the Assumption \ref{AssumptionInitialValueReps}.
\end{enumerate}

\end{remark}

\begin{corollary}\label{StrongConvergenceRepsBoundary}
For all $p\in [1,\infty)$ it holds up to a subsequence that
\begin{align*}
\epsilon^{-1} \reps \rightarrow R_0 \quad \mbox{ strongly in the two-scale sense on } \geps \mbox{ in } L^p.
\end{align*}
\end{corollary}
\begin{proof}
We use the equivalent characterization of the strong two-scale convergence via the strong convergence of the unfolded sequence $\teps(\epsilon^{-1} \reps)$ in $L^p((0,T)\times \Omega \times \Gamma)$, see Lemma \ref{LemmaAequivalenzTSKonvergenzUnfolding} in the Appendix \ref{SectionTwoScaleConvergence}. Since $\reps$ is constant on every micro cell, $\teps(\reps)$ is constant with respect to $y \in Y$ and therefore $\nabla_y \teps (\reps) = 0$. Hence, we have with the trace inequality 
\begin{align*}
\Vert &\teps(\epsilon^{-1} \reps) - R_0 \Vert_{L^p((0,T)\times \Omega \times \Gamma)} 
\\
&\le C \bigg( \Vert \teps(\epsilon^{-1} \reps) - R_0 \Vert_{L^p((0,T)\times \Omega \times Y^{\ast})} + \Vert \underbrace{ \nabla_y (\teps(\epsilon^{-1}\reps) - R_0)}_{=0}\Vert_{L^p((0,T)\times \Omega \times Y^{\ast})} \bigg)
\\
&\le C \Vert \teps(\epsilon^{-1} \reps) - R_0 \Vert_{L^p((0,T)\times \Omega \times Y)} \overset{\epsilon\to 0}{\longrightarrow} 0,
\end{align*}
where the convergence follows from Proposition \ref{StrongConvergenceReps} and Lemma \ref{LemmaAequivalenzTSKonvergenzUnfolding}.
\end{proof}
Let us define
\begin{align*}
J_0(t,x,y) &:= \det\big(\nabla_y S_0(t,x,y)\big),
\\
D_0(t,x,y) &:= J_0(t,x,y) \nabla_y S_0(t,x,y)^{-1} D (x) \nabla_y S_0(t,x,y)^{-T}.
\end{align*}
For the definition of $S_0$ see Corollary \ref{WeakTSCompactnessResults}.
We emphasize that $0 \kl c_0 \le J_0(t,x,y)$ for a constant $c_0 \gr 0$ and almost every $(t,x,y) \in (0,T) \times \Omega \times Y^{\ast}$, and therefore $\nabla_y S(t,x,y)^{-1} $ exists almost everywhere. From the strong convergence of $\reps$, proved in Proposition \ref{StrongConvergenceReps}, we immediately obtain: 
\begin{corollary}\label{StrongTSConvergenceCoefficients}
For $1 \in [1,\infty)$ it holds up to a subsequence that
\begin{align*}
\nabla \seps &\rightarrow \nabla_y S_0 &\mbox{strongly in the two-scale sense in }& L^p,
\\
\jeps &\rightarrow J_0 &\mbox{strongly in the two-scale sense in }& L^p,
\\
\deps &\rightarrow D_0 &\mbox{strongly in the two-scale sense in }& L^p.
\end{align*}
Additionally, the strong two-scale convergence of $\nabla \seps$ and $\jeps$ is also valid on $\geps$ in $L^p$.
\end{corollary}
\begin{proof}
Again, we use the unfolding operator. We have
\begin{align*}
\Vert \teps (\nabla \seps) - \nabla_y S_0 \Vert_{L^p((0,T)\times \Omega \times Y^{\ast})} &= \Vert (\teps(\epsilon^{-1}\reps) - R_0 ) \chi_0 \nu_0 \Vert_{L^p((0,T)\times \Omega \times Y^{\ast})}
\\
&\le C\Vert (\teps(\epsilon^{-1}\reps) - R_0 )   \Vert_{L^p((0,T)\times \Omega \times Y^{\ast})}
\end{align*}
The right-hand side converges to zero for $\epsilon \to 0$ because of Proposition \ref{StrongConvergenceReps}. For $\jeps$ we use the essential boundedness of $\teps(\jeps)$ and $J_0$ on $(0,T)\times \Omega \times Y^{\ast}$  uniformly with respect to $\epsilon$, to obtain with the local Lipschitz continuity of the determinant (using $\teps(\jeps) = \det (\teps(\nabla \seps))$)
\begin{align*}
\Vert \teps (\jeps) - J_0 \Vert_{L^p((0,T)\times \Omega \times Y^{\ast})} &\le C \Vert \teps(\nabla \seps) - \nabla_y S_0\Vert_{L^p((0,T)\times \Omega \times Y^{\ast})}\overset{\epsilon \to 0}{\longrightarrow} 0.
\end{align*}
To show the convergence of $D_{\epsilon}$, we first notice that due to $\eqref{BoundInverseMatrix}$ the inverse gradients $\teps(\nabla \seps^{-1})$ and $\nabla_y S_0^{-1}$ are essential bounded on $(0,T)\times \Omega \times Y^{\ast})$ uniformly with respect to $\epsilon$. This gives 
\begin{align*}
\Vert \teps(\nabla \seps^{-1}) - \nabla_y S_0^{-1}\Vert_{L^p((0,T)\times \Omega \times Y^{\ast})} &\le C \Vert \teps (\nabla \seps) - \nabla_y S_0 \Vert_{L^p((0,T)\times \Omega \times Y^{\ast})}
\\
&\overset{\epsilon\to 0}{\longrightarrow} 0, 
\end{align*}
where we have used the straightforward equality  
\begin{align*}
A^{-1} - B^{-1} = A^{-1} (B-A) B^{-1},
\end{align*}
valid for any invertible matrices $A,B \in \R^{n\times n}$. 
Using now \eqref{DefVepsDeps}, the convergence above yields the strong two-scale convergence of $D_{\epsilon}$. 

The convergence results for $\nabla \seps $ and $\jeps$ on $\geps$ follow by similar arguments and using Corollary \ref{StrongConvergenceRepsBoundary}.
\end{proof}

\subsubsection{Strong convergence of $\ueps$}
\label{SubsectionStrongConvergenceUeps}

Next, we prove the strong convergence of $\ueps$ in the two-scale sense to the limit function $u_0$. Since we have no uniform bound of $\partial_t \ueps$ with respect to $\epsilon$, standard methods (see \cite{MeirmanovZimin}) fail. Therefore, we first introduce a regularized auxiliary problem:
We consider for $t \in (0,T)$ and $\phi_0 \in C^{\infty}([0,T]\times \overline{\Omega})$ the problem
\begin{align}
\begin{aligned}
\label{AuxiliaryProblem}
- \Delta \weps + \weps &= \jeps(\ueps - \phi_0) &\mbox{ in }& \oe,
\\
-  \nabla \weps \cdot \nu &= 0 &\mbox{ on }& \partial \oe
\end{aligned}
\end{align}
Obviously, there exists a unique weak solution $\weps \in L^2((0,T),H^1(\oe))$. We have the following additional regularity with respect to time and uniform \textit{a priori} estimates with respect to $\epsilon$.

\begin{lemma}\label{AprioriEstimateAuxiliaryEquation}
It holds that $\weps \in  H^1((0,T),H^1(\oe))$ with
\begin{align}\label{AprioriEstimateAuxiliaryEquationInequality}
\Vert \partial_t \weps \Vert_{L^2((0,T),H^1(\oe))} + \Vert \weps \Vert_{L^{\infty}((0,T),H^1(\oe))} \le C,
\end{align}
with a constant $C\gr 0 $ independent of $\epsilon$.
\end{lemma}
\begin{proof}
Testing the weak formulation of $\eqref{AuxiliaryProblem}$ with $\weps$ we obtain for a constant $c_0 \gr 0$ almost everywhere in $(0,T)$ with Lemma \ref{AprioriEstimatesSummary}
\begin{align*}
\Vert \weps \Vert_{L^2(\oe)}^2 + \Vert \nabla \weps \Vert_{L^2(\oe)}^2 &=  \int_{\oe} \jeps (\ueps -  \phi_0) \weps dx
\\
&\le \Vert \jeps \Vert_{L^{\infty}(  \oe)} \Vert \ueps - \phi_0 \Vert_{  L^2(\oe)} \Vert \weps \Vert_{L^2(\oe)} .
\\
&\le C \Vert \weps\Vert_{L^2(\oe)}
\\
&\le C + \frac12 \Vert \weps \Vert_{L^2(\oe)}^2.
\end{align*}
This implies 
\begin{align*}
\Vert \weps \Vert_{L^{\infty}((0,T),H^1(\oe))} \le C,
\end{align*}
where the constant $C$ is independent of $\epsilon$, but depends on the $L^{\infty}((0,T),L^2( \Omega))$-norm of $\phi_0$.

Let us check the estimate for the time-derivative. We define for $0 \kl h \ll T$ the difference quotient with respect to time for a function $\peps : (0,T)\times \oe \rightarrow \R$ by
\begin{align*}
\dth \peps(t,x) := \frac{\peps (t+h,x) - \peps(t,x)}{h} \quad \mbox{ for } (t,x)\in (0,T-h) \times \oe.
\end{align*}
Applying $\dth$ to the equation $\eqref{AuxiliaryProblem}$, we obtain
\begin{align*}
-\Delta \dth \weps + \dth \weps &= \dth[\jeps (\ueps - \phi_0)]   &\mbox{ in }& \Omega,
\\
-   \nabla \dth \weps \cdot \nu &=0 &\mbox{ on }& \partial \oe.
\end{align*}
Testing the weak formulation of the equation above with $\dth \weps$ and integrating with respect to time, we obtain
\begin{align*}
\int_0^{T-h}\int_{\oe}  \nabla \dth \weps \cdot \nabla &\dth \weps + \vert \dth \weps \vert^2 dxdt = \int_0^{T-h}\int_{\oe} \dth [\jeps (\ueps - \phi_0)] \dth \weps dxdt
\\
&\le \Vert \dth (\jeps(\ueps - \phi_0))\Vert_{L^2((0,T-h),H^1(\oe)')} \Vert \dth \weps \Vert_{L^2((0,T-h),H^1(\oe))} 
\\
&\le C(\theta) + \theta \Vert \dth \weps \Vert_{L^2((0,T-h),H^1(\oe))}^2
\end{align*}
for all $\theta \gr 0$ and a constant $C(\theta)\gr 0 $ depending on $\theta$, where in the last inequality we used the \textit{a priori} bounds for $\partial_t (\jeps \ueps)$, $\jeps$, and $\partial_t \jeps$  from Lemma \ref{AprioriEstimatesSummary}. The constant $C(\theta)$ depends on the $L^2((0,T)\times \Omega)$-norm of $\partial_t\phi_0$ and $\phi_0$. For $\theta$ small enough the last term can be absorbed from the left-hand side and we obtain
\begin{align*}
\Vert \dth \weps \Vert_{L^2((0,T-h),H^1(\oe))} \le C.
\end{align*}
This gives the desired result.
\end{proof}

\begin{remark}\label{RemarkWeps}\mbox{}
\begin{enumerate}
[label = (\roman*)]
\item By a density argument we want to choose $\phi_0 = u_0$. However, in the proof of the Lemma above we have to work with smooth $\phi_0$, since we need the time derivative $\partial_t (\jeps \phi_0)$ and it is not clear whether the time-derivative $\partial_t (\jeps u_0)$ exists and is bounded (not even clear $\partial_t u_0 \in L^2((0,T),H^1(\Omega)')$!).
\item The inequality $\eqref{AprioriEstimateAuxiliaryEquationInequality}$ depends on the norm
\begin{align*}
\Vert \phi_0 \Vert_{L^2((0,T),H^1(\Omega))} + \Vert \partial_t \phi_0 \Vert_{L^2((0,T)\times \Omega)}.
\end{align*}
In the proof of the following Proposition \ref{StrongConvergenceUeps} we will choose a sequence $\phi_0 = \phi_k$ which converges only in $L^2((0,T),H^1(\Omega))$, \ie the norm of the time-derivative $\partial_t u_0$ is in general not bounded. We will see in the proof that this has no influence on the result.
\item\label{RemarkWepsSobolevEmbedding} Due to the Sobolev-embedding and since $\oe$ is connected, we also have 
\begin{align*}
\Vert \weps \Vert_{L^2((0,T),L^q(\oe))} \le C
\end{align*}
for $q = \frac{2n}{n-2}$ if $n\neq 2$ and $q \in [1,\infty)$ if $n=2$.
\end{enumerate}

\end{remark}

\begin{corollary}\label{ConvergenceAuxiliarySequence}
There exists $w_0 \in L^2((0,T),H^1(\Omega))$  and an extension $\tilde{w}_{\epsilon} \in L^2((0,T),H^1(\Omega))$ of $\weps$, such that up to a subsequence it holds that
\begin{align*}
\tilde{w}_{\epsilon} &\rightharpoonup w_0 &\mbox{ weakly in }& L^2((0,T),H^1(\Omega)),
\\
\tilde{w}_{\epsilon} &\rightarrow w_0 &\mbox{ in }& L^2((0,T),L^2(\Omega)).
\end{align*}
Especially, we obtain the strong two-scale convergence of $\weps$ to $w_0$.
\end{corollary}
\begin{proof}
We refer to \cite{Acerbi1992} for the extension and for the strong convergence see \cite[Theorem 2.1]{MeirmanovZimin}  or \cite[Lemma 10]{Gahn}. 
\end{proof}
Obviously, $w_0 \in L^2((0,T),H^1(\Omega))$ solves the following macroscopic problem:
\begin{align}
\begin{aligned}
\label{AuxiliaryProblemLimitEquation}
- \nabla \cdot (A \nabla w_0) + w_0 &= \J_0 (u_0 - \phi_0) &\mbox{ in }& (0,T)\times \Omega,
\\
-A \nabla w_0 \cdot \nu &= 0 &\mbox{ on }& (0,T)\times \partial \Omega,
\end{aligned}
\end{align}
with 
\begin{align*}
\J_0(t,x) : = \int_{Y^{\ast}} J_0(t,x,y) dy \quad \mbox{ for almost every } (t,x) \in (0,T)\times \Omega,
\end{align*}
and $A \in \R^{n\times n}$ is defined by 
\begin{align*}
A_{ij} = \int_{Y^{\ast}} \big(\nabla \chi_i + e_i\big) \cdot \big(\nabla \chi_j + e_j \big) dy,
\end{align*}
and $\chi_i \in H^1_{\per}(Y^{\ast})/ \R$ are the solutions of the following cell problems:
\begin{align*}
- \nabla  \cdot \big(\nabla \chi_i + e_i\big) &=0 &\mbox{ in }& Y^{\ast},
\\
-\big(\nabla \chi_i + e_i\big) \cdot \nu &= 0 &\mbox{ on }& \Gamma,
\\
\chi_i \mbox{ is } Y\mbox{-periodic, } \int_{Y^{\ast}} \chi_i dy = 0.
\end{align*}

\begin{proposition}\label{StrongConvergenceUeps}
Up to a subsequence it holds that
\begin{align*}
\chi_{\oe} \ueps &\rightarrow \chi_{Y^{\ast}} u_0 &\mbox{ strongly in the two-scale sense in }& L^2,
\\
\ueps\vert_{\geps} &\rightarrow u_0 &\mbox{ strongly in the two-scale sense on }& \geps \mbox{ in } L^2.
\end{align*}
%
%
\end{proposition}

\begin{proof}
We choose a sequence $\phi_k \in C^{\infty}_0((0,T)\times \overline{\Omega})$, such that $\phi_k \rightarrow u_0$ in $L^2((0,T),H^1(\Omega))$ for $k\to \infty$.
We denote by $\weps^k$ the solutions of $\eqref{AuxiliaryProblem}$ for $\phi_0 = \phi_k$ and  $w_0^k$ is the solution of $\eqref{AuxiliaryProblemLimitEquation}$ for $\phi_0 = \phi_k$. 
Testing $\eqref{AuxiliaryProblem}$ with $\ueps - \phi_k$, we obtain with $0 \kl c_0 \le \jeps$
\begin{align*}
c_0 \Vert \ueps - \phi_k \Vert_{L^2( \oe)}^2 &\le \int_{\oe} \jeps (\ueps - \phi_k)^2 dx 
\\
&= \int_{\oe} \nabla (\ueps - \phi_k ) \cdot \nabla \weps^k dx +  \int_{\oe} (\ueps - \phi_k) \weps^k dx
\\
&\le \Vert \ueps - \phi_k \Vert_{H^1(\oe)} \Vert \weps^k\Vert_{H^1(\oe)}.
\end{align*}
Hence, we get
\begin{align}
\begin{aligned}
\label{AuxiliaryInequalityStrongConvergenceUeps}
\frac{c_0}{2} \Vert & \ueps - u_0 \Vert^2_{L^2((0,T)\times\oe)} \le c_0 \Vert \ueps - \phi_k\Vert_{L^2((0,T)\times\oe)}^2 + c_0\Vert u_0 - \phi_k\Vert^2_{L^2((0,T)\times\oe)}
\\
&\le \Vert \ueps - \phi_k \Vert_{L^2((0,T),H^1(\oe))} \Vert \weps^k\Vert_{L^2((0,T),H^1(\oe))} +  c_0\Vert u_0 - \phi_k \Vert^2_{L^2((0,T)\times\oe)}.
\end{aligned}
\end{align}
The second term on the right-hand side tends to zero for $k \to \infty$, due to the choice of $\phi_k$.
For the first term we  notice that due to the strong convergence of $\phi_k$ in $L^2((0,T),H^1(\Omega))$ and Lemma \ref{AprioriEstimatesSummary}, we have
\begin{align*}
 \Vert  \ueps -  \phi_k \Vert_{L^2((0,T),H^1(\Omega))} \le C
\end{align*}
for a constant $C\gr 0$ independent of $k$ and $\epsilon$. We have to estimate $\Vert \weps^k\Vert_{L^2((0,T) , H^1(\oe))}$. From $\eqref{AuxiliaryProblem}$ we obtain by testing with $\weps^k$ after integration with respect to time:
\begin{align*}
\Vert \weps^k &\Vert_{L^2((0,T),H^1(\oe))}^2 = \int_0^T \int_{\oe} \jeps (\ueps - \phi_k) \weps^k dx dt
\\
&= \int_0^T \jeps (\ueps - u_0)\weps^k dx dt + \int_0^T \int_{\oe} \jeps (u_0 - \phi_k) \weps^k dx dt
\\
&\le \int_0^T \int_{\oe} \jeps (\ueps - u_0) \weps^k dx dt 
\\
&\hspace{2em}+ \Vert \jeps\Vert_{L^{\infty}((0,T)\times \oe)} \Vert u_0 - \phi_k \Vert_{L^2((0,T)\times \oe)} \Vert \weps^k\Vert_{L^2((0,T)\times \oe)}
\\
&\le \int_0^T \int_{\oe} \jeps (\ueps - u_0) \weps^k dx dt + C \Vert u_0 - \phi_k\Vert_{L^2((0,T)\times \oe)}^2 +\frac12 \Vert \weps^k \Vert_{L^2((0,T)\times \oe)}^2.
\end{align*}
Hence, we obtain
\begin{align*}
\Vert \weps^k \Vert_{L^2((0,T),H^1(\oe))}^2 \le C \left(\int_0^T \int_{\oe} \jeps (\ueps - u_0) \weps^k dx dt + \Vert u_0 - \phi_k \Vert_{L^2((0,T)\times \oe)}^2 \right).
\end{align*}
We show that the first term on the right-hand side converges to zero for fixed $k$ and $\epsilon \to 0$. We have
\begin{align*}
\int_0^T \int_{\oe} \jeps (\ueps - u_0) \weps^k dx dt  =& \int_0^T \int_{\oe} \left(\jeps - J_0\left(x,\fxe\right)\right) (\ueps - u_0) w_0^k dx dt 
\\
&+ \int_0^T \int_{\oe} J_0\left(x,\fxe\right) (\ueps - u_0) w_0^k dx dt
\\
&+ \int_0^T \int_{\oe } \jeps (\ueps - u_0) (\weps^k - w_0^k) dx dt
\\
=&: A_{\epsilon}^1 + A_{\epsilon}^2 + A_{\epsilon}^3.
\end{align*}
For the third term $A_{\epsilon}^3$ we obtain with the \textit{a priori} estimates from Lemma \ref{AprioriEstimatesSummary} and the strong convergence of $\tweps^k$ from Corollary \ref{ConvergenceAuxiliarySequence}
\begin{align*}
\vert A_{\epsilon}^3 \vert \le \Vert\jeps \Vert_{L^{\infty}((0,T)\times \oe)} \Vert \ueps - u_0 \Vert_{L^2((0,T)\times \oe)} \Vert \weps^k - w_0^k\Vert_{L^2((0,T)\times \oe)} \overset{\epsilon\to 0}{\longrightarrow} 0.
\end{align*}
Since $J_0\left(x,\fxe\right) w_0^k$ is an admissible test-function in the two-scale sense, we obtain from the (weak) two-scale convergence of $\ueps $ from Corollary \ref{WeakTSCompactnessResults} that $A_{\epsilon}^2 \rightarrow 0$ for $\epsilon \to 0$. It remains to estimate $A_{\epsilon}^1$. From Remark \ref{RemarkWeps} we have $w_0^k \in L^2((0,T),L^q(\Omega))$ for $q \gr 2$. Hence, there exists $p \in (1,\infty)$ such that $\frac{1}{p} + \frac{1}{q} = \frac12$. From the H\"older-inequality we get with Lemma \ref{AprioriEstimatesSummary}
\begin{align*}
\vert A_{\epsilon}^1 \vert &\le C \left\Vert \jeps - J_0\left(x,\fxe\right) \right\Vert_{L^p((0,T)\times \oe)} \Vert \ueps - u_0 \Vert_{L^{\infty}((0,T),L^2(\oe))} \Vert w_0^k\Vert_{L^2((0,T),L^q(\oe))}
\\
&\le C  \left\Vert \jeps - J_0\left(x,\fxe\right) \right\Vert_{L^p((0,T)\times \oe)} \overset{\epsilon\to 0}{\longrightarrow} 0,
\end{align*}
where at the end we used the strong two-scale convergence of $\jeps $, see also Remark \ref{RemarkCharacterizationStrongTSConvergence} in the Appendix \ref{SectionTwoScaleConvergence}.
Altogether, we get: 
\begin{align*}
\limsup_{\epsilon\to 0} \Vert \weps^k \Vert_{L^2((0,T),H^1(\oe))} \le C \Vert u_0 - \phi_k\Vert_{L^2((0,T)\times \Omega)}.
\end{align*}
Altogether, we obtain from $\eqref{AuxiliaryInequalityStrongConvergenceUeps}$ 
\begin{align*}
\limsup_{\epsilon \to 0} \Vert \ueps - u_0\Vert_{L^2((0,T)\times \oe)} \le C \left(\Vert u_0 - \phi_k\Vert^2_{L^2((0,T)\times \Omega)} + \Vert u_0 - \phi_k\Vert_{L^2((0,T)\times \Omega)} \right).
\end{align*}
For $k\to \infty $ we get the first convergence in the statement. To prove the strong two-scale convergence on the surface $\geps$ we use the trace inequality $\eqref{ScaledTraceInequality}$ to obtain with Lemma \ref{AprioriEstimatesSummary}
\begin{align*}
\sqrt{\epsilon} \Vert \ueps - u_0 \Vert_{L^2((0,T)\times \geps)} &\le C \left( \Vert \ueps - u_0 \Vert_{L^2((0,T)\times \oe)} + \epsilon \Vert \nabla \ueps - \nabla u_0 \Vert_{L^2((0,T)\times \oe)}\right)
\\
&\le C \Vert \ueps - u_0 \Vert_{L^2((0,T)\times \oe)} + C \epsilon,
\end{align*}
which tends to $0$ for $\epsilon \to 0$.
\end{proof}

\subsubsection{Strong convergence of $\partial_t \reps$}

In this section we identify the limit of $\partial_t \reps$, \ie we show $\partial_t R_0 = g(u_0,R_0)$.

\begin{lemma}\label{LemmaStrongConvergenceDtReps}
Let $\weps \in L^1((0,T)\times \geps)$ and $w_0 \in  L^1((0,T),W^{1,1}(\Omega))$ such that $\weps \rightarrow w_0 $ strongly in the two-scale sense on $\geps$ in $L^1$.
Then we have
\begin{align*}
\feg \int_{\geps\left(\epsilon \left[\frac{\cdot}{\epsilon}\right] \right)} \weps (t,z) d\sigma \rightarrow w_0 \quad \mbox{ in } L^1((0,T)\times \Omega).
\end{align*}
\end{lemma}
\begin{proof}
First of all we extend the function $w_0$ to a function in $L^1((0,T),W^{1,1}(\R^n))$ and use the same notation $w_0$ for the extended function.  Now, we define for all $(t,x) \in (0,T)\times \Omega$ 
\begin{align*}
\heps(t,x):=  \feg \int_{\geps\left(\epsilon \left[\fxe\right]\right)} \weps(t,z) d\sigma.
\end{align*}
Hence, we have to show that
\begin{align*}
\heps \rightarrow w_0 \quad \mbox{ in } L^1((0,T)\times \Omega).
\end{align*}
We also define for $(t,x)\in (0,T)\times \Omega$ (see $\eqref{DefinitionGepsx}$ for the definition of $\geps(x)$)
\begin{align*}
g_{\epsilon}^0(t,x) &:= \feg \int_{\geps(x) } w_0(t,z) d\sigma,
\\
\heps^0(t,x) &:= g_{\epsilon}^0\left(t,\epsilon \left[\fxe\right]\right).
\end{align*}
Then we have
\begin{align*}
\Vert &\heps - w_0 \Vert_{L^1((0,T)\times \Omega)} 
\\
&\le \Vert \heps - \heps^0 \Vert_{L^1((0,T)\times \Omega)} + \Vert \heps^0 - g_{\epsilon}^0\Vert_{L^1((0,T)\times \Omega)} + \Vert g_{\epsilon}^0 - w_0 \Vert_{L^1((0,T)\times \Omega)}
\\
&=: A_{\epsilon}^1 + A_{\epsilon}^2 + A_{\epsilon}^3.
\end{align*}
For the first term $A_{\epsilon}^1$ we have
\begin{align*}
A_{\epsilon}^1 &= \int_0^T\int_{\Omega} \left\vert \feg \int_{\geps\left(\epsilon \left[\fxe\right] \right)} \weps(t,z) - w_0(t,z) d\sigma_z \right\vert dx dt 
\\
&\le \feg \sum_{k \in K_{\epsilon}} \int_0^T\int_{\epsilon (Y+k)} \int_{\geps(\epsilon k)} \vert \weps(t,z) - w_0(t,z)\vert dz dx dt
\\
&\le \frac{\epsilon^n}{\vert \epsilon \Gamma\vert} \int_0^T \int_{\geps} \vert \weps(t,z) - w_0(t,z)\vert d\sigma dt
\\
&\le C \epsilon \Vert \weps - w_0 \Vert_{L^1((0,T)\times \geps)} \overset{\epsilon \to 0 }{\longrightarrow} 0,
\end{align*}
due to the strong two-scale convergence of $\weps$, see also Lemma \ref{LemmaAequivalenzTSKonvergenzUnfolding} in the Appendix \ref{SectionTwoScaleConvergence}.
For the term $A_{\epsilon}^2$ we use a density argument together with the mean value theorem. More precisely, let $w_\ell \in C_0^{\infty}((0,T)\times \R^n)$ be a sequence with $w_\ell \rightarrow w_0$ for $\ell \to \infty$ in $L^1((0,T)\times \R^n)$. Remember that $w_0$ is extended by zero to the whole $\R^n$. We mention that we can also choose a sequence converging in $L^1((0,T),W^{1,1}(\Omega))$, but this is not necessary. We define
\begin{align*}
g_{\epsilon}^\ell(t,x) &:= \feg \int_{\geps(x)} w_\ell(t,z)d\sigma,
\\
\heps^\ell(t,x) &:= g_{\epsilon}^\ell\left(t,\epsilon \left[\fxe\right] \right).
\end{align*}
Then we have
\begin{align*}
A_{\epsilon}^2 
&\le \Vert \heps^0 - \heps^\ell \Vert_{L^1((0,T)\times \Omega)} + \Vert \heps^\ell - g_{\epsilon}^\ell\Vert_{L^1((0,T)\times \Omega)} + \Vert g_{\epsilon}^\ell - g_{\epsilon}^0\Vert_{L^1((0,T)\times \Omega)}
\\
&=: B_{\epsilon,\ell}^1 + B_{\epsilon,\ell}^2 + B_{\epsilon,\ell}^3.
\end{align*}
For the first term we get with similar decomposition arguments as for the term $A_{\epsilon}^1$ above and the trace inequality $\eqref{ScaledTraceInequality}$
\begin{align*}
B_{\epsilon,\ell}^1 &\le C \epsilon \Vert w_\ell - w_0 \Vert_{L^1((0,T)\times \geps)}
\\
&\le C \left( \Vert w_0 - w_\ell \Vert_{L^1((0,T)\times \Omega)} + \epsilon \Vert \nabla w_0 - \nabla w_\ell \Vert_{L^1((0,T)\times \Omega)} \right).
\end{align*}
The first term on the right-hand side converges to $0$ for $\ell \to \infty$ uniformly with respect to $\epsilon$. The second term on the right-hand side goes to $0$ for $\epsilon \to 0$ (for fixed $\ell \in \N$). For the  term $B_{\epsilon,\ell}^2$ we  get
\begin{align*}
B_{\epsilon,\ell}^2 & \le \feg \int_0^T\int_{\Omega} \int_{\geps(x)} \left\vert w_\ell\left(t,z - \epsilon\left\{\fxe\right\}\right) - w_\ell(t,z) \right\vert d\sigma_z dx dt
\\
&\le \frac{\epsilon}{\vert \epsilon \Gamma\vert} \int_0^T \int_{\Omega } \int_{\geps(x)} \int_0^1 \left\vert \nabla w_\ell \left(t,z - s \ \epsilon\left\{\fxe\right\} \right)\right\vert ds d\sigma_z dx dt
\\
&\le C \epsilon \Vert \nabla w_\ell \Vert_{C^0([0,T]\times \overline{\Omega})}.
\end{align*}
Hence, $B_{\epsilon,\ell}^2 \rightarrow 0$ for $\epsilon \to 0 $ for every fixed $\ell \in \N$.
For $B_{\epsilon,\ell}^3$ we have
\begin{align*}
B_{\epsilon,\ell}^3 &\le \feg \int_0^T \int_{\Omega} \int_{\geps(0)}  \vert w_\ell (t,x + z ) - w_0(t,x+z)\vert d\sigma_z dx dt
\\
&= \feg \int_0^T\int_{\geps(0)} \int_{\Omega + z } \vert w_\ell (t,x) - w_0(t,x)\vert dx d\sigma_z dt
\\
&\le C \Vert w_\ell - w_0 \Vert_{L^1((0,T)\times \R^n)},
\end{align*}
and therefore $B_{\epsilon,\ell}^3 \rightarrow 0$ for $\ell \to \infty$, uniformly with respect to $\epsilon$.

It remains to estimate the term $A_{\epsilon}^3$. With the notation from above we obtain
\begin{align*}
A_{\epsilon}^3 &\le \Vert g_{\epsilon}^0 - g_{\epsilon}^\ell\Vert_{L^1((0,T)\times \Omega)} + \Vert g_{\epsilon}^\ell - w_\ell \Vert_{L^1((0,T)\times \Omega)} + \Vert w_\ell - w_0 \Vert_{L^1((0,T)\times \Omega)}.
\end{align*}
The first term was already considered above, and the last term obviously tends to $0$ for $\ell \to \infty$. For the second term, we argue as for $B_{\epsilon,\ell}^2$ to obtain 
\begin{align*}
\Vert g_{\epsilon}^\ell - w_\ell \Vert_{L^1((0,T)\times \Omega)} &\le \feg  \int_0^T\int_{\Omega} \int_{\geps(x) } \vert w_\ell(t,z) - w_\ell(t,x) \vert d\sigma_z dx dt
\\
&\le C\epsilon \Vert \nabla w_\ell \Vert_{C^0([0,T]\times \overline{\Omega})}.
\end{align*}
Hence, $A_{\epsilon}^3 $ goes to zero for $\epsilon \to 0$.
 Altogether, we obtain the desired result.
\end{proof}

As an immediate consequence we obtain the strong convergence of $\epsilon^{-1} \partial_t \reps$:
\begin{proposition}\label{StrongConvergenceDtReps}
Up to a subsequence $\epsilon \to 0$, it holds that
\begin{align*}
\epsilon^{-1} \partial_t\reps \rightarrow \partial_t R_0 \quad \mbox{in } L^p((0,T)\times \Omega)
\end{align*}
for $p \in [1,\infty)$, with $R_0$ satisfying the equation 
\begin{align*}
\partial_t R_0 = g(u_0,R_0) \quad \mbox{a.e. in } (0,T)\times \Omega.
\end{align*}
Especially, we have up to a subsequence
\begin{align*}
\epsilon^{-1} \partial_t \reps \rightarrow \partial_t R_0 \quad \mbox{ strongly in the two-scale sense on } \geps \mbox{ in } L^p.
\end{align*}
\end{proposition}
\begin{proof}
The convergence of $\epsilon^{-1} \partial_t \reps$ in $L^1((0,T)\times \Omega)$ is a direct consequence of Lemma \ref{LemmaStrongConvergenceDtReps} by choosing
\begin{align*}
\weps(t,x) &:= g \left(\ueps(t,x), \epsilon^{-1} \reps(t,x)\right),
\\
w_0(t,x) &:= g(u_0(t,x),R_0(t,x)).
\end{align*}
The strong convergence of $\ueps$ and $\epsilon^{-1} \reps$, see Proposition \ref{StrongConvergenceReps} and \ref{StrongConvergenceUeps}, and the Lipschitz continuity of $g$  imply the strong two-scale convergence of $\weps $ on $\geps$ in $L^1$. Further it holds that $w_0 \in L^2((0,T),H^1(\Omega))$, due to the Lipschitz-continuity of $g$ and the regularity of $u_0$ and $R_0$, see Proposition \ref{StrongConvergenceReps}. Then, Lemma \ref{LemmaStrongConvergenceDtReps} implies the convergence of $\epsilon^{-1} \partial_t \reps$ in $L^1((0,T)\times \Omega)$, and the result for arbitrary $p \in [1,\infty)$ follows again from the dominated convergence theorem of Lebesgue. 
Now, from $\partial_t R_0 = g(u_0,R_0)$, we obtain 
\begin{align*}
\nabla \partial_t R_0 = \partial_1 g(u_0,R_0) \nabla u_0 + \partial_2 g(u_0,R_0) \nabla R_0 \in L^2((0,T)\times \Omega)
\end{align*}
Using that $\reps$ is constant on each micro cell and the trace inequality $\eqref{ScaledTraceInequality}$, we obtain with the same arguments as in the proof of Corollary \ref{StrongConvergenceRepsBoundary}
\begin{align*}
\sqrt{\epsilon}\Vert \epsilon^{-1} \partial_t \reps - &\partial_t R_0 \Vert_{L^2((0,T)\times \geps)} 
\\
&\le C\left(\Vert \epsilon^{-1} \partial_t \reps - \partial_t R_0 \Vert_{L^2((0,T)\times \oe)} + \sqrt{\epsilon} \Vert \nabla \partial_t R_0\Vert_{L^2((0,T)\times \oe)}\right)
\\
&\overset{\epsilon\to 0}{\longrightarrow} 0.
\end{align*}
Arguing again as in Corollary \ref{StrongConvergenceRepsBoundary}, the convergence is also valid for all $p\in[1,\infty)$, and the proposition is proved.
\end{proof}

\begin{corollary}\label{StrongTSConvergenceDtJeps}
Up to a subsequence $\epsilon \to 0$, one has for every $p \in [1,\infty)$ that 
\begin{align*}
\partial_t \jeps \rightarrow \partial_t J_0 \qquad\mbox{strongly in the two-scale sense in } L^p.
\end{align*}
\end{corollary}
\begin{proof}
This is a direct consequence of the Jacobi formula $\eqref{JacobiFormula}$ and the strong two-scale convergence results in Corollary \ref{StrongTSConvergenceCoefficients} and Proposition \ref{StrongConvergenceDtReps}.
\end{proof}

\subsection{Derivation of the macroscopic equation}

In this section we derive the macroscopic model in the limit $\epsilon \to 0$. In a first step, we derive the cell problems on the reference element $Y^{\ast}$. In the second step, we derive the macroscopic equation on $\Omega$. We emphasize that we already derived the limit equation for the radius in Proposition \ref{StrongConvergenceDtReps}. 

First of all, choose $\peps(t,x)= \epsilon \phi\left(t,x,\fxe\right)$ for $\phi \in C_0^{\infty}((0,T)\times \Omega, C_{\per}^{\infty}(Y^{\ast}))$ as a test-function in $\eqref{WeakFormulationMicroscopicModel}$. Due to our \textit{a priori} estimates in Lemma \ref{AprioriEstimatesSummary}, all the terms except the diffusion term are of order $\epsilon$. Hence, for $\epsilon \to 0$ we obtain from the strong two-scale convergence of $\deps$ from Corollary \ref{StrongTSConvergenceCoefficients} and the two-scale convergence of $\nabla \ueps$ from Corollary \ref{WeakTSCompactnessResults}
\begin{align*}
\int_0^T\int_{\Omega}\int_{Y^{\ast}} D_0(t,x,y) \big[\nabla_x u_0(t,x) + \nabla_y u_1(x,y)\big] \cdot \nabla_y \phi(t,x,y) dy dx dt = 0.
\end{align*}
From the linearity of the problem we obtain for $u_1$ the representation
\begin{align*}
u_1(t,x,y) = \sum_{i=1}^n \partial_{x_i} u_0(t,x) w_i(t,x,y),
\end{align*}
where $w_i \in L^2((0,T)\times \Omega,H^1(Y^{\ast})/\R)$ are the solutions of the cell problems
\begin{align}
\begin{aligned}
\label{CellProblems}
- \nabla_y \cdot \left( D_0 \big(\nabla_y w_i + e_i\big)\right) &= 0 \qquad \mbox{ in } (0,T)\times \Omega \times Y^{\ast},
\\
- D_0 \big(\nabla_y w_i + e_i \big) \cdot \nu &= 0 \qquad \mbox{ on } (0,T)\times \Omega \times \Gamma,
\\
w_i \mbox{ is } Y\mbox{-periodic, } \int_{Y^{\ast}}& w_i(t,x,y) dy = 0 \quad \mbox{ f.a.e. } (t,x) \in (0,T)\times \Omega.
\end{aligned}
\end{align}
We emphasize that the $L^{\infty}$-regularity of $D_0$ with respect to $(t,x)$ implies the $L^{\infty}$-regularity of $\nabla_y w_i$ with respect to $(t,x)$.

Next, we derive the macroscopic equation. We choose $\peps  = \phi$ with $\phi \in C^{\infty}_0\big((0,T)\times \overline{\Omega}\big)$ as a test-function in $\eqref{WeakFormulationMicroscopicModel}$ and integrate with respect to time: 
\begin{align}
\begin{aligned}
\label{DerivationMacProVarEquation}
\int_0^T\langle &\partial_t (\jeps \ueps) , \phi  \rangle_{H^1(\oe)',H^1(\oe)} dt  - \int_0^T\int_{\oe} \ueps \partial_t \jeps \phi dxdt \\
& - \int_0^T\int_{\oe} \veps \cdot \nabla \ueps \phi dx  dt0
+ \int_0^T\int_{\oe} \deps \nabla \ueps \cdot \nabla \phi dxdt \\ = \int_0^T& \int_{\oe} \jeps f_{\epsilon} \phi dxdt - \int_0^T\int_{\geps} \partial_t \reps (\ueps - \rho ) \jeps \phi  d\sigma dt.
\end{aligned}
\end{align}
Let us pass to the limit in every single term. For the term including the time-derivative we use the following fact: First of all, using the strong  two-scale convergence of $\jeps$ from Corollary \ref{StrongTSConvergenceCoefficients} and $\ueps$ from Proposition \ref{StrongConvergenceUeps}, it is easy to check that $\jeps \ueps$ converges (even strongly by the dominated convergence theorem of Lebesgue \cite[Theorem 3.25]{AltLFAenglisch}) in the two-scale sense to $J_0 u_0$ in $L^2$. Then, due to the \textit{a priori} estimate of $\jeps \ueps$ in Lemma \ref{AprioriEstimatesSummary}, we obtain that for the zero extension $\widetilde{\jeps \ueps}$ of $\jeps \ueps$ to the whole domain $\Omega$ it holds that (see also Lemma \ref{TwoScaleTimeDerivative} in the appendix)
\begin{align*}
\partial_t \left(\widetilde{\jeps \ueps} \right) \rightharpoonup \partial_t \big(\bar{J}_0 u_0\big) \quad \mbox{ weakly in } L^2((0,T),H^1(\Omega)'),
\end{align*}
with
\begin{align*}
\bar{J}_0 (t,x):= \int_{Y^{\ast}} J_0(t,x,y) dy.
\end{align*}
\begin{remark}\label{RegularityBarJ0}
Since $J_0$ is  polynomial in $R_0$ with coefficients depending on $y$ in a smooth way, we easily obtain that $\bar{J}_0 \in L^{\infty}((0,T),H^1(\Omega))$.
\end{remark}
Now, we obtain
\begin{align*}
\int_0^T\langle \partial_t (\jeps \ueps) , \phi  \rangle_{H^1(\oe)',H^1(\oe)} dt &= \int_0^T\left\langle \partial_t \left( \widetilde{\jeps \ueps}\right), \phi \right\rangle_{H^1(\Omega)',H^1(\Omega)} dt
\\
& \overset{\epsilon \to 0}{\longrightarrow} \int_0^T \langle \partial_t (\bar{J}_0 u_0) , \phi \rangle_{H^1(\Omega)',H^1(\Omega)} dt.
\end{align*}
We emphasize that $\bar{J}_0$ is the volume of the moving cell $Y^{\ast}(t,x)$ defined by
\begin{align*}
Y^{\ast}(t,x):= Y \setminus \overline{B_{R_0(t,x)}(x)},
\end{align*}
\ie we have $J_0(t,x) = \vert Y^{\ast}(t,x)\vert$.
For the second term on the left-hand side in $\eqref{DerivationMacProVarEquation}$ we use the strong two-scale convergence of $\ueps$ from Proposition \ref{StrongConvergenceUeps}, as well as the  two-scale convergence of $\partial_t \jeps $ to $\partial_t J_0$, see Proposition \ref{StrongTSConvergenceDtJeps}, to obtain
\begin{align*}
\int_0^T\int_{\oe} \ueps \partial_t \jeps \phi dx dt \overset{\epsilon \to 0 }{\longrightarrow} \int_0^T \int_{\Omega}\int_{Y^{\ast}} u_0(t,x) \partial_t J_0(t,x,y) \phi(t,x) dy dx dt.
\end{align*}
Since the $L^2$-norm of $\veps$ is of order $\epsilon$, the third term on the left-hand side in $\eqref{DerivationMacProVarEquation}$ vanishes for $\epsilon \to 0$. Using the strong two-scale convergence of $\deps$ from Corollary \ref{StrongTSConvergenceCoefficients} and the two-scale convergence of $\nabla \ueps $ from Corollary \ref{WeakTSCompactnessResults}, we obtain
\begin{align*}
\lim_{\epsilon\to 0 } \int_0^T \int_{\oe} \deps \nabla \ueps \cdot \nabla \phi dx dt &= \int_0^T\int_{\Omega} \int_{Y^{\ast}} D_0  \big[\nabla u_0  + \nabla_y u_1\big] \cdot \nabla \phi dy dx dt 
\\
&= \int_0^T\int_{\Omega} D_0^{\ast}(t,x) \nabla u_0 \cdot \nabla \phi dx dt ,
\end{align*}
with the homogenized diffusion coefficient $D_0^{\ast} \in L^{\infty}((0,T)\times \Omega)^{n\times n}$ defined by
\begin{align}\label{DefinitionHomogenizedDiffusion}
\big(D_0(t,x)^{\ast}\big)_{ij} := \int_{Y^{\ast}} D_0(t,x,y) \big[\nabla_y w_i(t,x,y) + e_i \big]\cdot \big[ \nabla_y w_j(t,x,y) + e_j\big] dy,
\end{align}
where $w_i$ are the solutions of the cell problems $\eqref{CellProblems}$.
For the first term on the right-hand side of $\eqref{DerivationMacProVarEquation}$ we have
\begin{align*}
\lim_{\epsilon \to 0 } \int_0^T \int_{\oe} \jeps f_{\epsilon} \phi dx dt  = \int_0^T \int_{\Omega} \int_{Y^{\ast}} J_0 f_0 \phi dy dx dt,
\end{align*}
where we used the two-scale convergence of $\jeps$ and $\seps$, and the continuity of $f$, see Assumption \ref{AssumptionRechteSeitef}. It remains to pass to the limit in the boundary term in $\eqref{DerivationMacProVarEquation}$. We only consider in more detail the term including $\ueps$.
The strong two-scale convergence of $\epsilon^{-1} \partial_t \reps $ and $\jeps $ on $\geps $ in $L^p $ from Proposition \ref{StrongConvergenceDtReps} and Corollary \ref{StrongTSConvergenceCoefficients}, as well as the strong two-scale convergence of $\ueps$ on $\geps $ in $L^2$ from Proposition \ref{StrongConvergenceUeps}  imply the strong two-scale convergence of the product $\epsilon^{-1} \partial_t \reps \ueps \jeps$ on $\geps $ in $L^q$ for $q = \frac{2p}{4  + p} \in (1, 2)$ (with $p \gr 4$). In fact, this is an immediate consequence of the characterization of the two-scale convergence via the unfolding operator, see Lemma \ref{LemmaAequivalenzTSKonvergenzUnfolding} in the Appendix \ref{SectionTwoScaleConvergence}. Hence, we obtain
\begin{align*}
\lim_{\epsilon\to 0 }\int_0^T &\int_{\geps}  \partial_t \reps (\ueps - \rho) \jeps \phi  d\sigma dt = \int_0^T \int_{\Omega} \partial_t R_0 (u_0  - \rho) \phi  \int_{\Gamma}J_0   d\sigma_y dx dt.
\end{align*}

Let us define the moving cell surface $\Gamma(t,x)$ by 
\begin{align*}
\Gamma(t,x):= \partial B_{R_0(t,x)}(\m).
\end{align*}
Using the identity $\eqref{FormulaGradientS0}$ we obtain $ \vert \nabla_y S_0^{-T}\nu\vert = \vert \nu\vert = 1$ on $\Gamma$, and we get that
\begin{align*}
\int_{\Gamma} J_0    d\sigma_y = \int_{\Gamma(t,x)} d\sigma = \vert \Gamma(t,x)\vert
\end{align*}
is the surface area of $\Gamma(t,x)$.
Altogether we obtain in the limit $\epsilon \to 0$:
\begin{align*}
\int_0^T \langle \partial_t (\bar{J}_0 u_0 ) ,& \phi \rangle_{H^1(\Omega)',H^1(\Omega)} dt - \int_0^T \int_{\Omega} \partial_t \bar{J}_0 u_0 \phi dx dt + \int_0^T \int_{\Omega} D_0^{\ast} \nabla u_0 \cdot \nabla \phi dx dt 
\\
&= \int_0^T \int_{\Omega} \int_{Y^{\ast}} J_0 f_0 dy \phi dx dt - \int_0^T \int_{\Omega} \partial_t R_0 (u_0 - \rho) \phi \vert \Gamma(t,x)\vert dx dt
\end{align*}
for all $\phi \in C_0^{\infty}((0,T)\times \overline{\Omega})$ and by density for all $\phi \in L^2((0,T),H^1(\Omega))$. Using the relation $\partial_t J_0 = \nabla_y \cdot V_0$ and the notation $q:= \int_{Y^{\ast}} \nabla_y \cdot V_0 dy$, we obtain the the variatioanl equation $\eqref{MacroModelVar}$. 
To finish the proof of Theorem \ref{MainResultConvergenceMacroModel} we have to establish the initial condition $u(0) = \uin$, where we have to show $u^0 \in C^0([0,T],L^2(\Omega)$. Of course, due to Lemma   \ref{TwoScaleTimeDerivative}, it holds that  $(\bar{J}_0 u_0)(0) = \bar{J}_0 \uin$.
The spatial regularity of $R_0$ from Proposition \ref{StrongConvergenceReps}, the equality $\partial_t R_0 = g(u_0,R_0) $ almost everywhere in $(0,T) \times \Omega$ and the boundedness of $g$ imply 
\begin{align*}
    R_0 \in C^{0,1}([0,T],L^{\infty}(\Omega)) \cap L^{\infty}((0,T),H^1(\Omega)).
\end{align*}
Together with the positivity of $J_0$, we obtain
\begin{align*}
    \bar{J}_0^{-1} \in C^{0,1}([0,T],L^{\infty}(\Omega)) \cap L^{\infty } ((0,T),H^1(\Omega)).
\end{align*}
Hence, for $q \gr n $ we obtain
\begin{align*}
     \bar{J}_0^{-1} \partial_t (\bar{J}_0 u_0) - \bar{J}_0^{-1} u_0 \partial_t \bar{J}_0 \in  L^2((0,T),W^{1,q}(\Omega)').
\end{align*}
Now, the product rule implies 
\begin{align*}
    \partial_t u_0 = \bar{J}_0^{-1} \partial_t (\bar{J}_0 u_0) - \bar{J}_0^{-1} u_0 \partial_t \bar{J}_0 \in L^2((0,T),W^{1,q}(\Omega)').
\end{align*}
 Especially, due to \cite[Lemma 7.3]{Roubicek}, we have $u_0 \in C^0([0,T],L^2(\Omega))$ and the initial condition $u_0 (0) \in L^2(\Omega)$ makes sense.  Using $(\bar{J}_0 u_0)(0) = \bar{J}_0 \uin$ and $\bar{J}_0 \in C^0([0,T],L^{\infty}(\Omega)$, we get the initial condition $u_0(0) = \uin$.

\begin{remark}\label{RemarkMainResult}\mbox{}
\begin{enumerate}
[label = (\roman*)]
\item The proof above gives the following regularity results 
\begin{align*}
\bar{J}_0 &\in C^{0,1}([0,T],L^{\infty}(\Omega)) \cap L^{\infty } ((0,T),H^1(\Omega)),
\\
\partial_t u_0 &\in L^2((0,T),W^{1,q}(\Omega)')
\end{align*}
for $q \gr n$.
\item The macroscopic equation is formulated on the fixed domain $\Omega$. The moving boundary is in the cell problems. In fact, if we define
\begin{align*}
\tilde{w}_i(t,x,y):= w_i(t,x,S_0^{-1}(t,x,y)) \quad \mbox{for } (t,x,y) \in (0,T)\times \Omega \times Y^{\ast}(t,x),
\end{align*}
then $\tilde{w}_i$ solves the following problem on the moving cell $Y^{\ast}(t,x)$ 
\begin{align*}
-\nabla_y \cdot \big(D (\nabla_y \tilde{w}_i + \nabla_y (S_0^{-1})^Te_i \big) &= 0 &\mbox{ in }& Y^{\ast}(t,x),
\\
-D (\nabla_y \tilde{w}_i + \nabla_y (S_0^{-1})^T e_i) \cdot \nu &= 0 &\mbox{ on }& \Gamma(t,x),
\\
\int_{Y^{\ast}(t,x)} J_0^{-1}(t,x,S_0^{-1}(t,x,y))\tilde{w}_i(t,x,y)  dy &= 0, &  \tilde{w}_i & \mbox{ is } Y\mbox{-periodic}.
\end{align*}
The evolution of $Y^{\ast}(t,x)$ is given by $S_0$ and therefore coupled via the ODE for $R_0$ to the macroscopic solution $u_0$.
\end{enumerate}
\end{remark}

\section{Conclusion}
\label{SectionConclusion}

We analyzed a reaction-diffusion problem in an evolving micro-domain depending on the solution of the equation, and derived a macroscopic model. 
The homogenized model is only depending on the macroscopic variable $x \in \Omega$. The information about the evolving microstructure are contained in the effective coefficients. The homogenized diffusion coefficient is given via solutions of cell problems on evolving reference element, depending on the limit functions $u_0$ and $R_0$. Hence, in every macroscopic point $x \in \Omega$, we have to solve a cell problem depending on the solution itself, what leads to a strongly coupled problem which numerical treatment is highly challenging.

We emphasize that the methods in this paper are not restricted to the scalar case and linear reaction kinetics. In fact, the results simply extend to systems with Lipschitz-continuous reaction rates for example of Michaelis-Menton-type.

Our results depend highly on the assumption that the evolution of the surface is given by an ordinary equation. In general, one has to consider for example an hyperbolic level set problem coupled to the transport equation, and even global-in-time solutions for the microscopic solutions are not guaranteed and even cannot be expected. However, for the treatment of more realistic applications on has to take into account such kind of models.

In our results we made no statements about the uniqueness of the micro- and the macro-model. For the homogenization process, uniqueness for the upscaled model is important to obtain the convergence of the whole sequence, which is not guaranteed in our case. However, the low regularity of the product $\bar{J}_0 u_0$ and $\partial_t u_0$ causes trouble for the application of standard energy arguments. In the microscopic model, especially the nonlinear  boundary term  including the time-derivative $\partial_t \reps$ makes things complicated. Here, one should take into account entropy methods, see for example \cite{MalekNecasRokytaRuzicka1996, Otto1996}.

From a physical point of point of view we would expect nonegativity and essential boundedness for a solution. An upper bound can be obtained for the micro-model under additional assumptions on the data, e.g.,  boundedness of $\uepsin$ and $f=0$ (or $f$ depending on $\ueps$ with suitable growth conditions). Nonnegativity for the micro-model can be obtained for growing grains, \ie $\partial_t \reps \geq 0$. In the present model, using the  average over the  boundary of the perforation in the differential equation of $\reps$ complicates the proof for nonnegativity, as there is no  pointwise relation between $\ueps$ and $\partial_t \reps$. On the other hand, this structure guarantees that the perforations remain radially symmetric. 
The treatment of these problems are part of our ongoing work.

\section*{Acknowledgments}
The research of both authors was supported by the Research Foundation-Flanders (FWO), Belgium, through the Odysseus programme (project G0G1316N). M. Gahn was also supported by SCIDATOS (Scientific Computing for Improved Detection and Therapy of Sepsis), a  project funded by the Klaus Tschira Foundation, Germany (Grant Number 00.0277.2015).

\begin{appendix}
\section{Some elemental calculations}
\label{SectionAppendixHanzawaReferenceElement}

Recalling \eqref{eq:Ru}, we choose a fixed $\delta_0 \in \big(0, \min\{ \frac12 - \oR , \oR \}\big)$ 
and define the symmetric cut-off function $\chi \in C_0^{\infty}(\R))$ such that 
\begin{align*}
0 \le \chi(z) & \le 1 \quad  \text{ for all } z \in \R, 
\\
\chi(z) & = 1 \quad \mbox{ if } |z| \le \frac{\delta_0}{2},
\\
\chi(z) & = 0 \quad \mbox{ if } |z| \geq \delta_0,
\\
z \chi'(z)& < 0 \quad \mbox{ if } |z| \in \left(\frac{\delta_0}{2}, \delta_0\right). 
\end{align*}
Further, with $\Gamma = \partial B_{\oR}(\m)$ being the sphere of radius $\oR$ centered in $\m$,  we let $d_{\Gamma}$ denote the signed-distance function to $\Gamma$ 
\begin{align*}
d_{\Gamma} (y) = \begin{cases}
\mathrm{dist} (y,\Gamma) &\mbox{ if } \vert y \vert \geq \oR,
\\
- \mathrm{dist}(y,\Gamma) &\mbox{ if } \vert y \vert < \oR.
\end{cases}
\end{align*}
Clearly, for any $y \in \R^n$ one has $d_{\Gamma} (y) = \|y - \m \|- \oR$  which is a smooth function for all arguments $y \neq \m$. 

With this, we consider the function 
\begin{align}\label{eq:chi0}
\chi_0: Y \rightarrow \R, \quad \chi_0(y) = \chi(d_{\Gamma}(y)), 
\end{align}
which is $Y$-periodic and smooth (also in $\m$, since it vanishes in a neighborhood of $\m$), having a compact support in $Y$. 

Further, for any $y \in \Gamma$ we denote by $\nu_0(y)$ the unit normal in $y$ to $\Gamma$ pointing outwards $B_{\oR}(\m)$. In this simplified setting we have $\nu_0(y) = \frac{y - \m}{|y - \m|} = \nabla d_{\Gamma}(y)$, and we use the same expression to extend $\nu_0$ to the set $Y\setminus \{\m\}$.

With $r$ s.t. $0<  r < \frac12$ we let $Y_r^{\ast}$ be the perforated cell 
\begin{align*}
Y_r^{\ast} := Y \setminus \overline{B_r(\m)}.  
\end{align*}
Recalling that the radius in \eqref{MicroscopicModel_ODE} is time-dependent, given the function $R_0: [0,T] \rightarrow \big[\uR,\oR\big]$ we define the Hanzawa transform $S_0:[0,T] \times Y_{\oR}^{\ast} \rightarrow Y$ by
\begin{align*}
S_0(t,y):= y + \big(R_0(t) - \oR\big) \chi_0(y) \nu_0(y).
\end{align*} 
Clearly, an equivalent form of $S_0$ reads
\begin{align*}
	S_0(t,y):= \m + \frac 1 {\|y - \m\|} \big[\|y - \m\| + \big(R_0(t) - \oR\big) \chi(\|y-\m\| -\oR) \big](y-\m).
\end{align*} 

As will be seen below, for any $t \in [0,T]$,  $S_0(t,\cdot): Y_{\oR}^{\ast} \rightarrow Y_{R_0(t)}^{\ast}$ is a bijective mapping. Note that the function $S_0$ is defined on the entire cell $Y$, but is only bijective in a neighborhood of $\Gamma$.

With $I_n$ denoting the unit matrix in $\R^{n\times n}$, the derivatives of $S_0$ are \begin{align*}
\nabla S_0(t,y) &= I_n + \big(R_0(t) - \oR\big) \left[ \nu_0(y) \otimes  \nabla_y \chi_0(y) + \chi_0 (y) \nabla_y \nu_0(y) \right],
\\
\partial_t S_0(t,y) &= R_0'(t) \chi_0(y) \nu_0(y).
\end{align*}
Also, for the functions $\chi_0$ and $\nu_0$ we have
\begin{align*}
\nabla \chi_0(y) &=  \frac{y-\m}{|y - \m|} \chi'\big(|y - \m| - \oR\big) = \nu_0(y)  \chi'\big(|y - \m| - \oR\big),
\\
\nabla \nu_0(y) &= \frac{1}{|y - \m|} I_n - \frac{1}{|y - \m|^3} (y - \m ) \otimes (y - \m) = \frac{1}{|y - \m|} I_n - \frac{1}{|y - \m|} \nu_0(y) \otimes \nu_0(y). 
\end{align*}
This gives 
\begin{align}\label{FormulaGradientS0}
\nabla S_0 = \left(1 + \frac{\chi_0}{|y- \m|} (R_0 - \oR) \right) I_n + \left[(R_0 - \oR) \left(\chi'(|y-\m| - \oR) - \frac{\chi_0}{|y-\m|} \right) \right] \nu_0 \otimes \nu_0 .
\end{align}

To compute the Jacobian determinant of $S_0$ we use the matrix determinant lemma (see \cite{Harville}, Theorem 13.3.8., or \cite{DING2007}) stating that 
\begin{align*}
\det (A + u \otimes v) = (1 + v^T A^{-1} u) \det (A),
\end{align*}
for any invertible matrix $A\in \R^{n\times n}$ and column vectors $u, v\in \R^n$. 
After an elemental calculation, one has  
\begin{align*}
\det(\nabla S_0(t,y)) = \left(1 + \chi'(|y-\m| - \oR) (R_0(t) - \oR) \right) \cdot \left( 1 + \frac{\chi_0(y)}{|y-\m|} (R_0(t) - \oR) \right)^{n-1}.
\end{align*}
Using the definition of $\chi$ and with $R_0 \in [\uR,\oR]$, for all $y \in Y_{\oR}^{\ast}$ one has 
\begin{align*}
1 \le 1 + \chi'(|y-\m| - \oR) (R_0 - \oR) \le 1 + \Vert \chi'\Vert_{L^{\infty}(\R)}\big(\oR - \uR\big),
\\
0 < \frac{\uR}{\oR} \le  \frac{R_0}{\oR} \le  1 + \frac{\chi_0(y)}{|y-\m|} (R_0 - \oR)  \le 1,
\end{align*}
as $\vert y - \m \vert \geq \oR$. Therefore
\begin{align*}
0 <\left(\frac{\uR}{\oR}\right)^{n-1} \le \det(\nabla S_0(t,y)) \le 1 + \|\chi'\|_{L^{\infty}(\R)}\big(\oR - \uR\big).
\end{align*}
In particular, this shows that $S_0(t,\cdot)$ is a bijection from $Y_{\oR}^{\ast}$ to  $Y_{R_0(t)}^{\ast}$.

\section{Two-scale convergence and unfolding}
\label{SectionTwoScaleConvergence}
In this section we briefly summarize the concept of the two-scale convergence and the unfolding operator. These methods provide the basic techniques to pass to the limit $\epsilon \to 0$ in the microscopic problem.

\subsection{Two-scale convergence}
\label{SubsectionTwoScaleConvergence}
We start with the definition of  the two-scale convergence, which was first introduced and analyzed in \cite{Nguetseng} and \cite{Allaire_TwoScaleKonvergenz}, see also \cite{LukkassenNguetsengWallTSKonvergenz}.

\begin{definition}\label{DefinitionTSConvergence}
A sequence $\ueps \in  L^p((0,T)\times \Omega)$ for $p\in [1,\infty)$ is said to converge  in the two-scale sense in $L^p$ to the limit function $u_0\in L^p((0,T)\times  \Omega \times Y)$, if for every $\phi \in L^{p'}((0,T)\times \Omega, C_{\per}(Y))$  the following relation holds
\begin{align*}
\lim_{\epsilon\to 0}\int_0^T \int_{\Omega}u_{\epsilon}(t,x)\phi\left(t,x,\frac{x}{\epsilon}\right)dxdt = \int_0^T\int_{\Omega}\int_Y u_0(t,x,y)\phi(t,x,y)dydxdt .
\end{align*}
A  two-scale convergent sequence $u_{\epsilon}$ convergences strongly in the two-scale sense to $u_0$, if  
\begin{align*}
\lim_{\epsilon\to 0}\|u_{\epsilon}\|_{L^p((0,T)\times \Omega)} =\|u_0\|_{L^p((0,T)\times \Omega  \times Y)} .
\end{align*}
\end{definition} 

\begin{remark}\label{RemarkCharacterizationStrongTSConvergence}
Let $u_0 \in L^p((0,T)\times \Omega,C_{\per}^0(Y))$. Then $\ueps $ converges strongly in the two-scale sense to $u_0$ in $L^p$ if and only if 
\begin{align*}
\lim_{\epsilon\to 0}\left\Vert \ueps - u_0 \left(x,\fxe\right)\right\Vert_{L^p((0,T)\times \Omega)}  = 0.
\end{align*}
\end{remark}

In \cite{AllaireDamlamianHornung_TwoScaleBoundary,Neuss_TwoScaleBoundary} the method of two-scale convergence was extended to oscillating surfaces:
\begin{definition}\label{DefinitionTSKonvergenzBoundary}
A sequence of functions $\ueps  \in L^p((0,T)\times\Gamma_{\epsilon})$ for $p \in [1,\infty)$ is said to converge  in the two-scale sense on the surface $\Gamma_{\epsilon}$ in $L^p$ to a limit $u_0\in L^p((0,T)\times \Omega \times \Gamma)$, if for every $\phi \in C\left([0,T]\times\overline{\Omega},C_{per}(\Gamma)\right)$ it holds that
\begin{align*}
\lim_{\epsilon \to 0} \epsilon \int_0^T\int_{\Gamma_{\epsilon}} u_{\epsilon}(t,x)\phi\left(t,x,\frac{x}{\epsilon}\right)d\sigma dt = \int_0^T\int_{\Omega}\int_{\Gamma} u_0(t,x,y)\phi(t,x,y)d\sigma_ydxdt .
\end{align*}

We say a  two-scale convergent sequence $u_{\epsilon}$ converges strongly in the two-scale sense, if additionally it holds that
\begin{align*}
\lim_{\epsilon\to 0}  \epsilon^{\frac{1}{p}}\|u_{\epsilon}\|_{L^p((0,T)\times \Gamma_{\epsilon})} = \|u_0\|_{L^p((0,T)\times \Omega \times \Gamma)} .
\end{align*}
\end{definition}


We have the following compactness results (see e.g., \cite{Allaire_TwoScaleKonvergenz,LukkassenNguetsengWallTSKonvergenz,Neuss_TwoScaleBoundary}:
\begin{lemma}\label{BasicTwoScaleCompactness}\
For every $p \in (1,\infty)$ we have:
\begin{enumerate}
[label = (\roman*)]
\item For every bounded sequence $\ueps \in L^p((0,T)\times \Omega)$ there exists $u_0 \in L^p((0,T)\times \Omega \times Y)$ such that up to a subsequence
\begin{align*}
\ueps &\rightarrow u_0 \qquad \mbox{in the two-scale sense in } L^p.
\end{align*}
\item For every bounded sequence $\ueps \in L^p((0,T),W^{1,p}(\Omega))$ there exist $u_0 \in L^p((0,T)\times \Omega)$ and $u_1 \in L^p((0,T)\times \Omega,W^{1,p}_{\per}(Y)/\R)$, such that up to a subsequence
\begin{align*}
\ueps &\rightarrow u_0 &\mbox{ in the two-scale sense in } L^p,
\\
\nabla \ueps &\rightarrow \nabla_x u_0 + \nabla_y u_1 &\mbox{ in the two-scale sense in } L^p.
\end{align*}
\item For every sequence $\ueps \in L^p((0,T),W^{1,p}(\Omega))$ with $\ueps $ and $\epsilon\nabla \ueps$ bounded in $L^p((0,T)\times \Omega)$, there exists $u_0 \in L^p((0,T)\times \Omega, W^{1,p}_{\per}(Y))$ such that up to a subsequence it holds that
\begin{align*}
    \ueps &\rightarrow u_0 &\mbox{ in the two-scale sense in }& L^p,
    \\
    \epsilon \nabla \ueps &\rightarrow \nabla_y u_0 &\mbox{ in the two-scale sense in }& L^p.
\end{align*}
\item For every sequence $\ueps \in L^p((0,T)\times \geps)$ with
\begin{align*}
\epsilon^{\frac{1}{p}}\Vert \ueps\Vert_{L^p((0,T)\times \geps)} \le C,
\end{align*}
there exists $u_0 \in L^p((0,T)\times \Omega \times \Gamma)$ such that up to a subsequence
\begin{align*}
\ueps \rightarrow u_0 \quad \mbox{ in the two-scale sense on } \geps \mbox{ in } L^p.
\end{align*}
\end{enumerate}

\end{lemma}

Concerning the time derivative we have the following result:
\begin{lemma}
\label{TwoScaleTimeDerivative}
Let $\weps \in L^2((0,T),H^1(\oe))\cap H^1((0,T),H^1(\oe)')$ with 
\begin{align*}
    \Vert \partial_t\weps \Vert_{L^2((0,T),H^1(\oe)')} + \Vert \weps \Vert_{L^2((0,T)\times \oe)} + \epsilon \Vert \nabla\weps\Vert_{L^2((0,T)\times \oe)} \le C.
\end{align*}
Denote by $\tweps $ the zero extension to the whole domain $\Omega$. Let $w_0 \in L^2((0,T)\times \Omega, H_{\per}^1(Y))$ denote the two-scale limit of $\tweps$ (up to a subsequence) from  Lemma \ref{BasicTwoScaleCompactness} (vanishing on $Y\setminus Y^{\ast}$) and define $\bar{w}_0:= \int_{Y} w dy$. Then, again up to a subsequence, it holds that
\begin{align*}
    \partial_t \tweps \rightarrow \partial_t \bar{w}_0 \quad \mbox{ weakly in } L^2((0,T), H^1(\Omega)').
\end{align*}
If additionally $\tweps(0)$ converges in the two-scale sense to $w^0 \in L^2(\Omega \times Y)$, then we have
\begin{align}
    \bar{w}_0(0) = \int_Y w^0 dy.
\end{align}
\end{lemma}
\begin{proof}
This follows by standard two-scale arguments and integration by parts in time, so we skip the details.
\end{proof}

\subsection{The unfolding operator}

When dealing with nonlinear problems it is helpful to work the unfolding method which gives a characterization for the weak and strong convergence, see Lemma \ref{LemmaAequivalenzTSKonvergenzUnfolding} below. For a detailed investigation of the unfolding operator and its properties we refer to \cite{CioranescuGrisoDamlamian2018}. For a perforated domain (here we also allow the case $Y^{\ast} = Y$, \ie $\oe = \Omega$) we define the unfolding operator for $p \in [1,\infty]$ by
\begin{align*}
\teps : L^p((0,T)\times \oe) \rightarrow L^p((0,T)\times \Omega \times Y^{\ast}),\quad \teps(\peps)(t,x,y)= \peps\left(t,\epsilon \left[\fxe\right] + \epsilon y\right).
\end{align*}
In the same way, we define the boundary unfolding operator for the oscillating surface $\geps$ via
\begin{align*}
\teps : L^p((0,T)\times \geps) \rightarrow L^p((0,T)\times \Omega \times \Gamma),\quad \teps(\peps)(t,x,y)= \peps\left(t,\epsilon \left[\fxe\right] + \epsilon y\right).
\end{align*}
We emphasize that we use the same notation for the unfolding operator on $\oe$ and the boundary unfolding operator $\geps$. We summarize the basic properties of the unfolding operator, see \cite{CioranescuGrisoDamlamian2018}:

\begin{lemma}\label{LemmaPropertiesUnfoldingOperator}
Let $p \in [1,\infty]$.
\begin{enumerate}[label = (\roman*)]
\item   For $\peps \in L^p((0,T)\times \oe)$ it holds that
\begin{align*}
\| \teps (\peps) \|_{L^p((0,T)\times \Omega \times \oe)} = \|\peps\|_{L^p((0,T)\times Y^{\ast})}.
\end{align*}
\item For $\peps \in L^p((0,T),W^{1,p}(\oe))$ it holds that
\begin{align*}
\nabla_y \teps (\peps )= \epsilon \teps( \nabla_x \peps).
\end{align*} 
\item  For $\peps \in L^p((0,T)\times \geps)$ it holds that
\begin{align*}
\|\teps \peps \|_{L^p((0,T)\times \Omega \times \Gamma)} = \epsilon^{\frac{1}{p}} \|\peps\|_{L^p((0,T)\times \geps)}.
\end{align*}
\end{enumerate}
\end{lemma}

The following lemma gives a relation between the unfolding operator and the two-scale converges. Its proof is quite standard and we refer the reader to \cite{BourgeatLuckhausMikelic} and \cite{CioranescuGrisoDamlamian2018} for more details (see also \cite[Proposition 2.5]{visintin2006towards}).
\begin{lemma}\label{LemmaAequivalenzTSKonvergenzUnfolding}
Let $p \in [1,\infty)$.
\begin{enumerate}
[label = (\alph*)]
\item For   a sequence $\ueps \in L^p((0,T)\times \Omega)$, the following statements are equivalent:
\begin{enumerate}[label = (\roman*)]
\item $\ueps \rightarrow u_0$ weakly/strongly in the two-scale sense in $L^p$,
\item $\teps \ueps \rightarrow u_0$ weakly/strongly in $L^p((0,T)\times \Omega \times Y)$.
\end{enumerate}
\item  For a sequence $\ueps \in L^p((0,T)\times \geps)$  the following statements are equivalent:
\begin{enumerate}[label = (\roman*)]
\item $\ueps \rightarrow u_0$ weakly/strongly in the two-scale sense on $\geps$ in $L^p$,
\item $\teps \ueps \rightarrow u_0$ weakly/strongly in $L^p((0,T)\times \Omega \times \Gamma)$.
\end{enumerate}
\end{enumerate}
\end{lemma}

\end{appendix}

\bibliographystyle{abbrv}
\bibliography{literature}

\end{document}